\documentclass[twoside,12pt]{article}
\usepackage{import}
\usepackage{preamble}
\usepackage{microtype}
\usepackage{geometry}
\usepackage{fancyhdr}
\fancypagestyle{plain}{
  \fancyhf{}
  \fancyfoot[C]{\small\thepage}

}

\title{Relative dualizability and the cobordism hypothesis for defects}
\author{William Stewart}
\date{}

\excludecomment{claudia}

\begin{document}

\maketitle
\begin{abstract}
    This paper compares two candidate notions of morphism in the setting of fully local topological field theory. A first description can be formulated in terms of defects of codimension $k$ separating a pair of defects of codimension $(k-1)$. Lurie's cobordism hypothesis with singularities identifies such defects with suitably dualizable $k$-morphisms in the target category. A second description can be formulated in terms of symmetric monoidal (op)lax natural transformations between the truncations of the source and target theories. We show that these two descriptions agree.

    The comparison rests on a criterion for when a $k$-morphism determines a codimension-$k$ defect, expressed as a one-sided iterated adjunctibility condition. Assuming the ordinary cobordism hypothesis, this comparison provides a reformulation of the cobordism hypothesis for defects in terms of systems of oplax natural transformations.
    
    We also apply the criterion to the higher Morita category, where it yields a reducibility theorem for modules over $E_n$-algebras: a module over an $(n+1)$-dualizable $E_n$-algebra $A$ is $n$-dualizable over $A$ if and only if its underlying $E_{n-1}$-algebra is $n$-dualizable.
\end{abstract}

\tableofcontents

\pagestyle{fancy}

\section{Introduction}

The cobordism hypothesis \cite{BD95,L09} classifies fully local $n$-dimensional topological field theories in terms of sufficiently dualizable objects in the target category. Moreover, Lurie extends the cobordism hypothesis to include defects: a codimension-$k$ defect is classified in terms of sufficiently dualizable $k$-morphisms in the target category. In this paper, we give equivalent conditions for a $k$-morphism to define a codimension-$k$ defect. We use this result to connect different approaches to defects, and to obtain a reducibility theorem for modules over $E_n$-algebras.

We will focus primarily on the following prototypical defects. Consider a defect of codimension $k$ whose local structure corresponds to $\RR^n$ with stratification by the subspaces $\{0\}^k  \times \RR^{n-k} \subset \{0\}^{k-1}\times \RR^{n-k+1} \subset \cdots \subset \RR^n$. For $n=3$, and $k=1,2,3$, these local models are

\begin{center}
\vspace{2mm}
\begin{tikzpicture}[
  x={(0.92cm,-0.20cm)},
  y={(0cm,0.92cm)},
  z={(0.55cm,0.28cm)},
  ambient edge/.style={
    draw=black!35,
    line width=0.35pt
  },
  wall/.style={
    draw=wallcolour!80!black,
    fill=wallcolour,
    fill opacity=0.18,
    line width=0.55pt
  },
  line defect/.style={
    draw=linecolour!85!black,
    line width=2.1pt,
    line cap=round
  },
  point defect/.style={
    circle,
    fill=pointcolour!85!black,
    draw=white,
    line width=0.45pt,
    inner sep=2.25pt
  }
]
% NOTE: the role of the x/z coordinates below should be switched to match the convention ({0}^k x R^{n-k} \subset ... \subset R^n).

% #1 = horizontal displacement
% #2 = deepest codimension to draw
\newcommand{\prototypicaldefect}[2]{%
  \begin{scope}[xshift=#1cm]

    % The codimension-one stratum {0} x R^2.
    \path[wall]
      (-1.45,-1.05,0) --
      ( 1.45,-1.05,0) --
      ( 1.45, 1.05,0) --
      (-1.45, 1.05,0) -- cycle;

    % A box indicating a neighbourhood in the ambient R^3.
    \draw[ambient edge]
      (-1.45,-1.05,-0.85) --
      ( 1.45,-1.05,-0.85) --
      ( 1.45, 1.05,-0.85) --
      (-1.45, 1.05,-0.85) -- cycle;

    \draw[ambient edge]
      (-1.45,-1.05,0.85) --
      ( 1.45,-1.05,0.85) --
      ( 1.45, 1.05,0.85) --
      (-1.45, 1.05,0.85) -- cycle;

    \foreach \x/\y in {
      -1.45/-1.05,
       1.45/-1.05,
       1.45/ 1.05,
      -1.45/ 1.05
    }
      \draw[ambient edge]
        (\x,\y,-0.85) -- (\x,\y,0.85);

    % For k >= 2, add the codimension-two stratum
    % {0}^2 x R.
    \ifnum#2>1
      \draw[line defect]
        (-1.45,0,0) -- (1.45,0,0);
    \fi

    % For k = 3, add the codimension-three stratum {0}^3.
    \ifnum#2>2
      \node[point defect] at (0,0,0) {};
    \fi

    \node[font=\small] at (0,-1.8,0) {$k=#2$};
  \end{scope}%
}

\prototypicaldefect{0}{1}
\prototypicaldefect{4.55}{2}
\prototypicaldefect{9.10}{3}

\end{tikzpicture}
\end{center}

In \Cref{sec:def-structure}, we give a detailed account of the tangential structures and labeling systems with which these prototypical defects can be equipped. Each label for a prototypical defect of codimension $k$ comes with a source and target label for the codimension-$(k-1)$ defects on either side of that stratum. In this way, these prototypical defects give a notion of a map between defects of one less codimension. For example, a prototypical codimension-one defect---a domain wall---gives an interface between a pair of $n$-dimensional framed field theories---the bulk theories---on either side of the domain wall.

A second notion of a morphism between topological field theories can be formulated in terms of (op)lax natural transformations. This idea originates in the twisted field theories of Stolz--Teichner \cite{ST11} and in the relative field theories of Freed--Teleman \cite{FT12}, each of which presents a theory as a transformation out of the trivial theory. A definition of this transformation as a symmetric monoidal (op)lax natural transformation was given by Johnson-Freyd--Scheimbauer \cite{JFS15}. Given a pair of $n$-dimensional field theories $F_A$ and $F_B$, one can consider an (op)lax natural transformation $\beta:\tau_{\leq n-1} F_A \Rightarrow \tau_{\leq n-1}F_B$ between their truncations to dimension $(n-1)$.

In this paper, we reconcile these two notions in the setting of fully local field theory. In the case of a domain wall, we show the following:

\begin{introtheorem}\label{intro-thm-1}
Assuming the cobordism hypothesis for defects, an $n$-dimensional fully local domain wall is equivalent to a pair of $n$-dimensional fully local bulk theories $F_A$ and $F_B$, and a symmetric monoidal oplax natural transformation $\beta:\tau_{\leq n-1} F_A \Rightarrow \tau_{\leq n-1} F_B$ between their truncations.
\end{introtheorem}

\Cref{intro-thm-1} is also true, replacing oplax natural transformations by lax natural transformations. In \Cref{sec:defect-cob-hyp}, we generalize \Cref{intro-thm-1} to prototypical defects of arbitrary codimension, equipped with general tangential structures and labeling systems. A useful motto is that the interacting tangential structures on each stratum are encoded in the truncation map, while labeling systems correspond to systems of oplax natural transformations with source and target compatibility. We also provide a version of \Cref{intro-thm-1} for the general singularities introduced in \cite[Section 4.3]{L09}.

\begin{warning}
In general, a relative field theory is an (op)lax natural transformation from the trivial theory to an arbitrary once-categorified field theory $G$. In this paper, we consider only the case that $G$ arises as the truncation of a bulk theory defined in one degree higher. In the general setting, there is no equivalence between the oplax and lax formulations, and neither of these notions is captured by a fully local field theory with singularities in the sense of \cite{L09}. In \cite[Part~II]{StePhD}, the author gives a geometric definition of a bordism category such that functors from it correspond to relative field theories (in the more general sense); this construction, along with further motivating examples, will appear in \cite{Relative}. 
\end{warning}

Numerous constructions of boundary theories and defects in the literature proceed either through the cobordism hypothesis for defects or through symmetric monoidal (op)lax natural transformations. For example, the cobordism hypothesis for defects has been applied to gapped boundary theories \cite{FT21} and to extended Rozansky--Witten models with arbitrary defect networks \cite{BCFR24}. The approach through symmetric monoidal (op)lax natural transformations has been used to describe Witten--Reshetikhin--Turaev theory as a boundary condition for Crane--Yetter theory \cite{Hai23,Hai25} and to construct domain walls between finite gauge theories via higher integral transforms \cite{VD23b}. Further applications include the construction of a non-semisimple Crane--Yetter theory relative to classical gauge theory \cite{Kin24} and of full logarithmic conformal field theories from three-dimensional topological field theories with surface defects \cite{HR25}. Our results show that topological theories constructed through (op)lax natural transformations define fully local defect theories in the sense of \cite{L09} precisely when their bulk theories extend to the full ambient dimension.

The prototypical defects described above are expected to fit into a higher category, in which a codimension-$k$ defect is a $k$-morphism between the codimension-$(k-1)$ defects on either side of it. Categories of this shape appear throughout the study of topological symmetries, where the symmetries of a field theory are organized by the topological defects it admits \cite{GKSW15,BBSNT23,FMT22}, and in the study of defect topological field theories \cite{CRS19}. A general framework for such categories, for defects equipped with tangential structures, is proposed in \cite{Mul25}, and established there for fully extended topological field theories assuming the cobordism hypothesis with singularities. Our results suggest that such a category admits an equivalent description in which the $k$-morphisms are the symmetric monoidal oplax natural transformations of \cite{JFS15}, taken between the truncations of the corresponding source and target theories.

\begin{terminology}\label{term:defect-cob-hyp}
All of the statements we refer to as the \emph{cobordism hypothesis for defects} are special cases of Lurie's cobordism hypothesis with singularities \cite[Theorem~4.3.11]{L09}, stated as \Cref{hyp:general-singular-cobordism}. We use the phrase as an umbrella term for versions of the cobordism hypothesis with singularities that classify defects; for the prototypical defects considered here, it is \Cref{hyp:def-cob-hyp-main-2}.
\end{terminology} 

So far, we have described the two approaches to defects and the comparison between them. We now give an overview of the technical result underlying that comparison, the reformulation of the cobordism hypothesis for defects, and the application to modules over $E_n$-algebras.

\paragraph{Equivalent dualizability conditions.} The main technical result of this paper is a formulation of equivalent dualizability conditions for a $k$-morphism in a symmetric monoidal $(\infty,n)$-category $\Cc$. The comparison of different approaches to defects then follows from this result, using the cobordism hypothesis.

We say that a $k$-morphism $f$ is \emph{ambiently $n$-dualizable} if it belongs to a symmetric monoidal $(\infty,n)$-subcategory of $\Cc$ in which every object admits a dual and every $i$-morphism for $i<n$ admits both a left and right adjoint\footnote{Equivalently, $f$ belongs to the maximal such subcategory.}. The cobordism hypothesis for defects classifies framed codimension $k$-defects in terms of ambiently $n$-dualizable $k$-morphisms in $\Cc$. We use the term ambiently $n$-dualizable to communicate the idea that these morphisms define defects internal to ambient $n$-dimensional field theories.

In contrast, Johnson-Freyd--Scheimbauer classify oplax natural transformations in terms of a one-sided iterated adjunctibility condition. A $k$-morphism $f$ is \emph{$(n-k)$-times right-adjunctible} if $f$ admits a right adjoint and, recursively, all unit and counit $i$-morphisms arising in the resulting adjunction data admit right adjoints, for $i < n$.

We show that these dualizability conditions for a $k$-morphism coincide when the source and target are ambiently $n$-dualizable.

\begin{restatable}{introtheorem}{ThmMainDual}\label{thm:main-dual-result}
Let $f$ be a $k$-morphism whose source and target are ambiently $n$-dualizable. Then the following are equivalent:
\begin{enumerate}
    \item $f$ is ambiently $n$-dualizable.
    \item $f$ is $(n-k)$-times right-adjunctible.
\end{enumerate}
    The same is true when replacing `right' with `left', or replacing `$(n-k)$-times right-adjunctible' with any other one-sided iterated adjunctibility condition (see \Cref{def:dexterity-fn}).
\end{restatable}

\paragraph{Reformulating the cobordism hypothesis for framed domain walls.} Assuming the cobordism hypothesis for domain walls, \Cref{intro-thm-1} establishes an equivalence between a domain wall and a pair of bulk theories, along with an oplax natural transformation between their truncations. In fact, \Cref{thm:main-dual-result} gives a stronger result: assuming the ordinary cobordism hypothesis, the cobordism hypothesis for domain walls may be derived from the aforementioned equivalence. We now explain this result in the case of a framed domain wall. In \Cref{sec:defect-cob-hyp}, we extend this comparison to prototypical defects of arbitrary codimension, equipped with general tangential structures and labeling systems, and then to the arbitrary singularities given in \cite{L09}.

Heuristically, a framed domain wall describes a codimension-one interface between two $n$-dimensional framed field theories, called its \emph{bulk theories}. Its topology is encoded by a bordism category $\Bord_n^{\fr,\dom}$ whose bordisms are equipped with a codimension-one submanifold separating two disjoint regions labeled by $A$ and $B$. Along the interface, the last $n-1$ vectors of the $n$-framing are required to be tangent, and hence induce an $(n-1)$-framing. The first framing vector\footnote{An alternative convention uses the $n$-th framing vector to determine the coorientation, as in \cite{L09}. We use the first framing vector because this makes the interpretation of codimension-$k$ defects in terms of $k$-morphisms more apparent.} determines a coorientation of the interface and is required to point outward from the region labeled by $A$.

There are canonical symmetric monoidal functors
\begin{equation*}
\iota_A:\Bord_n^\fr\longrightarrow\Bord_n^{\fr,\dom}
\qquad\text{and}\qquad
\iota_B:\Bord_n^\fr\longrightarrow\Bord_n^{\fr,\dom},
\end{equation*}
obtained by regarding an ordinary framed bordism as a bordism with no interface and labeled entirely by $A$ or $B$, respectively. Given a symmetric monoidal functor $Z:\Bord_n^{\fr,\dom}\to\Cc$, its bulk theories are $F_A:=Z\circ\iota_A$ and $F_B:=Z\circ\iota_B$.

The \emph{framed domain wall interval} is the interval $[-1,1]$ with its standard $n$-framing and a single interface point at $0$, separating the region labeled by $A$ for $x<0$ from the region labeled by $B$ for $x>0$. For $n=2$, it is depicted as follows:
\begin{center}
\begin{tikzpicture}[
    scale=1.2,
    >=stealth,
    frame/.style={
        draw=framingblue,
        line width=0.7pt
    }
]
    % Regions of the interval
    \draw[defectpurple, line width=1.1pt]
        (-3,0) -- (0,0);
    \draw[defectyellow, line width=1.1pt]
        (0,0) -- (3,0);

    % Endpoints and interface
    \fill[defectpurple] (-3,0) circle (1.5pt);
    \fill[defectyellow] (3,0) circle (1.5pt);
    \fill[interfacegreen] (0,0) circle (2pt);

    % Constant 2-framing
    \foreach \x in {-3,-2,-1,0,1,2,3} {
        \draw[frame,->]
            (\x,0) -- ++(0.4,0);
        \draw[frame,->>]
            (\x,0) -- ++(0,0.55);
    }

    % Region labels
    \node[defectpurple, below=4pt] at (-1.5,0) {$A$};
    \node[defectyellow, below=4pt] at (1.5,0) {$B$};

    % Coordinate labels
    \node at (-3,-0.4) {$-1$};
    \node at (0,-0.4) {$0$};
    \node at (3,-0.4) {$1$};
\end{tikzpicture}
\end{center}

The framed domain wall interval defines a 1-morphism in $\Bord_n^{\fr,\dom}$ from the positively framed point labeled by $A$ to the positively framed point labeled by $B$. Consequently, a symmetric monoidal functor $Z:\Bord_n^{\fr,\dom}\to\Cc$ assigns to it a 1-morphism
\begin{equation*}
f:F_A(\pt_+)\longrightarrow F_B(\pt_+).
\end{equation*}
Since $\Bord_n^{\fr,\dom}$ has duals for objects, and adjoints for $i$-morphisms with $i<n$, the morphism $f$ is ambiently $n$-dualizable. The cobordism hypothesis for framed domain walls \cite{L09} asserts that this assignment gives an equivalence between symmetric monoidal functors $Z:\Bord_{n}^{\fr,\dom}\to \Cc$ and ambiently $n$-dualizable 1-morphisms $f:X\to Y$ in $\Cc$.

Being ambiently $n$-dualizable, the framed domain wall interval is, in particular, $(n-1)$-times right-adjunctible. Assuming the ordinary cobordism hypothesis, the classification of oplax natural transformations \cite{JFS15} associates to the framed domain wall interval a symmetric monoidal oplax natural transformation
\begin{equation*}
\begin{tikzcd}
{\Bord_{n-1}^\fr} && \Bord_n^{\fr,\dom}
\arrow[""{name=0, anchor=center, inner sep=0},
"{\tau_{\leq n-1}\iota_A}"{pos=0.475},
curve={height=-24pt}, from=1-1, to=1-3]
\arrow[""{name=1, anchor=center, inner sep=0},
"{\tau_{\leq n-1}\iota_B}"',
curve={height=24pt}, from=1-1, to=1-3]
\arrow["\beta_{\Id}", shorten <=6pt, shorten >=6pt,
Rightarrow, from=0, to=1]
\end{tikzcd}
\end{equation*}
corresponding to the framed domain wall interval. For a symmetric monoidal functor $Z:\Bord_n^{\fr,\dom}\to\Cc$, let $\beta_Z:=Z\beta_{\Id}$ denote the symmetric monoidal oplax natural transformation obtained by whiskering $\beta_{\Id}$ with $Z$.

The following theorem gives a reformulation of the cobordism hypothesis for framed domain walls. \Cref{thm:main-comparison} extends this result to prototypical defects of arbitrary codimension equipped with general tangential structures and labeling systems, while \Cref{thm:general-singular-comparison} extends it further to arbitrary singularities in the sense of \cite{L09}. 

\begin{introtheorem}\label{thm:framed-dom-wall-intro}
Assuming the ordinary cobordism hypothesis, the cobordism hypothesis for framed domain walls is equivalent to the assertion that the assignment $$Z\mapsto(Z \circ \iota_A,Z\circ \iota_B,\beta_Z)$$ induces an equivalence between the following types of data:
\begin{enumerate}[label=(\arabic*)]
\item Symmetric monoidal functors $Z:\Bord_n^{\fr,\dom}\to\Cc$.
\item Triples $(F_A,F_B,\beta)$ consisting of symmetric monoidal functors $F_A,F_B:\Bord_n^\fr\to\Cc$
and a symmetric monoidal oplax natural transformation
\begin{equation*}
\beta:\tau_{\leq n-1}F_A\Longrightarrow
\tau_{\leq n-1}F_B \ .
\end{equation*}
\end{enumerate}
The analogous statement holds with `oplax' replaced by `lax'.
\end{introtheorem}

Fix $F_A,F_B:\Bord_n^\fr\to\Cc$, and set $X:=F_A(\pt_+)$ and $Y:=F_B(\pt_+)$. All functors and natural transformations in the following diagram are understood to be symmetric monoidal. We summarize the comparison as follows:
\begin{equation*}\scalebox{0.75}{
\begin{tikzpicture}[node distance=4cm]
\node (domwall) [process] {$\begin{lrdcases}
    \text{Framed domain walls from $F_A$ to $F_B$}
\end{lrdcases}$};
\node (oplax) [process, right of=domwall,xshift=7cm] {$\begin{lrdcases}
    \text{oplax natural transformations} \\ \hspace{2mm} \text{$\beta: \tau_{\leq n-1} F_A\Longrightarrow \tau_{\leq n-1} F_B$}
\end{lrdcases}$};
\node (defectdual) [process, below of=domwall] {$\begin{lrdcases}
    \hspace{3mm} \text{ambiently $n$-dualizable} \\ \text{1-morphisms $f:X\to Y$}
\end{lrdcases}$};
\node (nrightdual) [process, right of=defectdual,xshift=7cm] {$\begin{lrdcases} \text{$(n-1)$-times right-adjunctible} \\ \quad \text{1-morphisms $f:X\to Y$}\end{lrdcases}$};
\draw[->] (domwall) -- (oplax) node [midway, label=above:$Z \mapsto \beta_{Z}$] {};
\draw [->] (domwall) -- (defectdual) node [midway, label=left: Cobordism hypothesis for] {} node [pos=.75, label=left:framed domain walls \cite{L09}] {};
\draw[->] (defectdual) -- (nrightdual) node [midway, label=above:$\simeq$] {} node [midway, label=below:\Cref{thm:main-dual-result}] {};
\draw[->] (oplax) -- (nrightdual) node [midway, label=right:\cite{JFS15} using ordinary] {} node [pos=0.75, label=right:cobordism hypothesis] {} node [midway, label=left:$\simeq$] {} ;
\end{tikzpicture} } \end{equation*}

\begin{remark}\label{rem:univ-prop}
There is also a dual universal-property perspective on \Cref{thm:framed-dom-wall-intro}. Maps out of $\Bord_n^{\fr,\dom}$ should be classified, naturally in the target, by a pair of $n$-dimensional framed bulk theories together with an oplax natural transformation between their truncations. Formulating this as a genuine universal property requires a higher-categorical setting in which symmetric monoidal higher categories, symmetric monoidal functors, and oplax natural transformations can be organized.
 
A framework for formulating such universal properties is provided by Masuda's theory of categorical spectra \cite{Mas24}, developed independently of and concurrently with the present work. Masuda constructs cobordism categorical spectra with singularities as certain iterated (co)extensions, built from the lax Gray tensor product. The universal properties exhibited by these objects are analogous to the reformulations presented here. A key step in the present paper is establishing that the relevant lax and oplax constructions agree in the presence of adjoints; the same expectation is noted in \cite[Remark~6.3.2]{Mas24}. It would be interesting to compare Masuda's categorical-spectrum models with $(\infty,n)$-categories constructed geometrically from bordisms equipped with defect data.
\end{remark}

\paragraph{Ongoing work: the geometric picture.} The results of this paper apply purely in the setting of fully local topological field theory, and the techniques we use involve manipulations of duality and adjunction data in higher categories. However, behind the scenes, geometric intuition of the bordism category has suggested these theorems and our reformulation of the cobordism hypothesis for defects. We briefly explain these ideas here, and refer the reader to \cite{StePhD} for further details. The realization of these ideas in an $\infty$-categorical setting is work in progress; carrying it out would provide a geometric proof of our main conjectures (\Cref{conj:main-conj} and \Cref{conj:general-singular-reformulation}), which, in this paper, we show are equivalent, upon assuming the ordinary cobordism hypothesis, to the cobordism hypothesis for defects. Hence, this would provide a geometric proof of the cobordism hypothesis for defects, assuming the ordinary cobordism hypothesis. Here, our starting point is a bordism category defined geometrically in analogy with the ordinary bordism category $\Bord_{n}$ \cite{L09,CS19}; a corresponding statement within the framework of categorical spectra is established in \cite{Mas24}.

Our reformulation of the cobordism hypothesis for domain walls says that a domain wall is determined by a pair of bulk theories and an oplax natural transformation between their truncation. The bulk theories have a clear geometric interpretation: they arise from restricting the domain wall to bordisms with no interface, equipped with a single bulk label. The oplax natural transformation also admits a geometric interpretation, it should be the \emph{dimensional reduction} along the domain wall interval $\Idom$. The difficulty with this definition is understanding what is meant by dimensional reduction along the domain wall interval.

Let $\Idom$ be the domain wall interval (for the geometric interpretation, it is not so important to have a framing). Recall that $\Idom$ is a bordism from $\pt_A$, the point labeled $A$, to $\pt_B$, the point labeled $B$. Suppose that $M$ is a closed manifold of dimension $k < n$. Then the product $M\times \Idom$ admits the structure of bordism from the manifold $M$ labeled $A$ to the manifold $M$ labeled $B$. Given a domain wall $Z:\Bord_{n}^{\dom}\to \Cc$, evaluation on $M \times \Idom$ gives a morphism $F_A(M) \to F_B(M)$ in $\Cc$, where $F_A$ and $F_B$ are the bulk theories for $Z$. This is the data that an oplax natural transformation from $\tau_{\leq n-1} F_A$ to $\tau_{\leq n-1}F_B$ assigns to $M$.

This geometric construction becomes more complicated when $M$ is a bordism with corners. In this case, $M \times \Idom$ is a manifold with corners, and it is not immediately clear how to interpret this manifold in the bordism category $\Bord_{n}^{\dom}$. When $M$ is a bordism between closed manifolds $X$ and $Y$, then $M \times \Idom$ has four faces $M\times \pt_A$, $M\times \pt_B$, $X \times \Idom$ and $Y \times \Idom$. By smoothing and inserting edges at the corners, the manifold with corners $M \times \Idom$ may be interpreted as a diagram of the form

\[\begin{tikzcd}
	{X\times \pt_A} && {X\times \pt_B} \\
	\\
	{Y\times \pt_A} && {Y\times \pt_B}
	\arrow["{X\times \Idom}", from=1-1, to=1-3]
	\arrow["{M\times \pt_A}"', from=1-1, to=3-1]
	\arrow["{M\times \pt_B}", from=1-3, to=3-3]
	\arrow["{M\times \Idom}"{description}, between={0.2}{0.8}, Rightarrow, nfold, from=3-1, to=1-3]
	\arrow["{Y\times \Idom}"', from=3-1, to=3-3]
\end{tikzcd}\]

Evaluating $Z$ on this diagram, then gives the data that an oplax natural transformation from $\tau_{\leq n-1} F_A$ to $\tau_{\leq n-1}F_B$ assigns to $M: X\to Y$. This unfolding procedure becomes more complicated when $M$ has corners of increasing codimension. A construction in the setting of bicategories, and a heuristic construction in the setting of $(\infty,n)$-categories is provided in \cite{StePhD}, utilizing notions of companions and conjoints in higher uple-categories. More generally, one could replace $\Idom$ by another bordism with or without defects.

The geometric intuition of dimensional reduction along $\Idom$ explains why one gets an oplax natural transformation from a domain wall. Since the construction is geometric, it would also apply in the non-fully local setting. To see why this data, along with the bulk theories, should determine a domain wall, we appeal to another very geometric idea. Given a manifold $M$ equipped with a codimension-one submanifold $M_1$, there exists a collar neighborhood of that submanifold (which is unique up to contractible choice). By taking such a collar neighborhood, a manifold, or more generally a bordism, can be decomposed into three pieces. One labeled entirely by $A$, one labeled entirely by $B$, and one corresponding to the collar, which is of the form $M_1 \times \Idom$. The value of a domain wall should be reconstructed by applying the bulk theories on the pieces labeled by $A$ and $B$, and by applying the oplax natural transformation to the component $M_1 \times \Idom$.

\paragraph{Application to $E_n$-modules.} Aside from its application to defects, \Cref{thm:main-dual-result} is of interest in its own right. Let $\Ss$ be a presentably symmetric monoidal $(\infty,1)$-category, and let $\Mor_n(\Ss)$ denote the higher Morita category of $E_n$-algebras in $\Ss$ \cite{Sch14,Kar25,SSS26}. For example, applying \Cref{thm:main-dual-result} and the results of \cite{SSS26} to a 1-morphism $M:A\to B$ in $\Mor_n(\Ss)$ yields the following result. Here $n$-dualizability of $M$ over $A$ and over $B$ are relative dualizability conditions, defined in \Cref{def:n-dualizable-over}; for $n=1$ they say that $M$ is dualizable as a left $A$-module, respectively as a right $B$-module, recovering the usual notions of \cite[Section~4.6]{LHA}.
 
\begin{restatable}{introcorollary}{ThmEnBimod}\label{thm:En-bimod}
Let $A$ and $B$ be $(n+1)$-dualizable $E_n$-algebras in $\Ss$ and $M$ an $(A,B)$-bimodule (i.e.\ a 1-morphism $M:A\to B$ in $\Mor_{n}(\Ss)$). Then the following are equivalent:
\begin{enumerate}
    \item $M$ is ambiently $(n+1)$-dualizable.
    \item $M$ is $n$-dualizable over $A$.
    \item $M$ is $n$-dualizable over $B$.
\end{enumerate}
\end{restatable}
 
In the case that $B$ is trivial, \Cref{thm:En-bimod} reduces to the following result, which we think of as a reducibility theorem for modules over an $(n+1)$-dualizable $E_{n}$-algebra. Here we use the term \emph{left module over $A$} for a 1-morphism $M:A\to\unit$ in $\Mor_n(\Ss)$.
 
\begin{restatable}{introcorollary}{ThmEnMod}\label{thm:En-mod}
Let $M$ be a left module over an $(n+1)$-dualizable $E_n$-algebra $A$ in $\Ss$. Then $M$ is $n$-dualizable over $A$ if and only if the $E_{n-1}$-algebra underlying $M$ is $n$-dualizable in $\Mor_{n-1}(\Ss)$.
\end{restatable}
 
\begin{example}
The case $n=1$ is established in \cite[Section~4.6]{LHA}. An $E_1$-algebra $A$ is 2-dualizable if and only if it is smooth and proper. In this case, \Cref{thm:En-mod} says that an $A$-module $M$ is dualizable over $A$ if and only if $M$ is dualizable as an object of $\Ss$.
\end{example}
 
\begin{example}
Suppose that $\Ss$ is the ordinary 1-category $\Vect_k$, where $k$ is a perfect field. A smooth and proper algebra object $A$ is then precisely a finite-dimensional semisimple $k$-algebra. \Cref{thm:En-mod} says that every $A$-module whose underlying vector space is dualizable---that is, finite-dimensional---is dualizable as an $A$-module. Equivalently, every finite-dimensional $A$-module is finitely generated and projective. Since $A$ is semisimple, this recovers the classical statement that every finite-dimensional $A$-module is completely reducible.
\end{example}

\paragraph{Oplax convention.}
We formulate our results primarily in terms of oplax natural transformations and right-adjunctibility. Under the relevant dualizability hypotheses, our main result identifies $(n-1)$-times right-adjunctibility with $(n-1)$-times left-adjunctibility; consequently, the corresponding formulations in terms of lax natural transformations are equivalent. One can also obtain formulations in terms of natural transformations corresponding to the general mixed adjunctibility conditions of \cite{SS23}.

We choose oplax as our primary type of natural transformation since, for any closed manifold $M$ of dimension $k<n$, an oplax natural transformation $\beta$ between the truncations of $F_A$ and $F_B$ assigns to $M$ a $(k+1)$-morphism $\beta_M : F_A(M)\to F_B(M)$. For a lax natural transformation $\alpha$, the direction of $\alpha_M$ is reversed for odd $k$. Thus, oplax most closely matches the geometric intuition of a domain wall from $F_A$ to $F_B$.

One reason to prefer lax natural transformations, noted in \cite{JFS15}, is that lax natural transformations $\alpha\colon\unit\Rightarrow\unit$ of the trivial theory $\unit\colon\Bord_{n}\to\Cc$ are equivalent to symmetric monoidal functors $\Bord_{n-1}\to \Omega\Cc:=\Hom_{\Cc}(\unit,\unit)$ \cite[Theorem~7.4]{JFS15}. The oplax analogue instead has target $(\Omega\Cc)^{\mathrm{odd}\text{-}\mathrm{op}}$, in which the direction of all $i$-morphisms with $i$ odd and $i\leq n-1$ is reversed. In the presence of duals and adjoints, these agree, so oplax natural transformations between trivial theories also recover theories of one dimension lower.

\paragraph{Outline of the paper.} The body of the paper is divided into three sections. \Cref{sec:dualizability} proves \Cref{thm:main-dual-result}, the comparison of dualizability conditions on which the other two sections rest. \Cref{sec:defect-cob-hyp} uses it to reformulate the cobordism hypothesis for defects, proving \Cref{intro-thm-1} and \Cref{thm:framed-dom-wall-intro}; to illustrate the roles of tangential structures and defect labels, we begin that section with a detailed treatment of the prototypical defects equipped with arbitrary tangential structures and labeling systems, before passing to the arbitrary singularities of \cite[Section~4.3]{L09}. \Cref{sec:En-application} applies \Cref{thm:main-dual-result} in the higher Morita category, proving \Cref{thm:En-bimod} and \Cref{thm:En-mod}.

\paragraph{Acknowledgments.} The majority of the ideas in this article were developed and presented in my PhD thesis \cite{StePhD}, especially Chapter 8. I would like to thank my advisor, Dan Freed, for his consistent encouragement and guidance throughout this project. I am also especially grateful to Claudia Scheimbauer for her enthusiasm for these ideas and for many helpful discussions. I would also like to thank Ben Haïoun, Lukas Müller, Charlie Reid, Pelle Steffens, Jackson van Dyke, and Richard Wedeen for many motivating and enlightening conversations regarding these ideas over the years. 

During this project, I have been supported by the Simons Collaboration on Global Categorical Symmetries (1013836 and 8528-03), and by the Deutsche Forschungsgemeinschaft  (DFG, German Research Foundation) through the Collaborative Research Center SFB 1085 Higher invariants - 224262486.

\section{Defect dualizability}\label{sec:dualizability}

In this section, we prove \Cref{thm:main-dual-result}, which establishes necessary and sufficient conditions for a $k$-morphism $f$ in a symmetric monoidal $(\infty,N)$-category $\Cc$ to be ambiently $n$-dualizable. In Sections \ref{sec:dual-defs}, \ref{sec:partial-duals} and \ref{sec:Araujo}, we review definitions and results from \cite{L09,Ara17,SS23} regarding duals and adjoints in higher categories. The proof of Theorem \ref{thm:main-dual-result} is provided in Section \ref{sec:def-dual-cond}.

\subsection{Definition of $n$-dualizability}\label{sec:dual-defs}

We first recall some definitions and results regarding duals and adjoints in symmetric monoidal $(\infty,N)$-categories.

\begin{definition}
Let $\Cc$ be a symmetric monoidal $(\infty,N)$-category. An object $X$ in $\Cc$ is dualizable if there exists an object $X^\vee \in \Cc$, along with 1-morphisms
\begin{equation*}
    \ev : X^\vee \otimes X \to \unit \quad \text{and} \quad \coev: \unit \to X \otimes X^\vee
\end{equation*}
in $\Cc$, and invertible 2-morphisms $z_1$ and $z_2$ in $\Cc$ exhibiting the zig-zag relations:
\begin{align*}
    X \xrightarrow{\simeq} \unit \otimes X \xrightarrow{\coev \otimes \Id_X} X \otimes X^\vee \otimes X \xrightarrow{\Id_X \otimes \ev } \unit \otimes X  \xrightarrow{\simeq} X\xRightarrow{z_1} \Id_X  
\end{align*}
and  
\begin{align*}
    {X^\vee} \xrightarrow{\simeq} {X^\vee} \otimes \unit\xrightarrow{\Id_{X^\vee} \otimes \coev} {X^\vee}\otimes X \otimes {X^\vee} \xrightarrow{\ev \otimes \Id_{X^\vee}} {X^\vee} \otimes \unit \xrightarrow{\simeq} X^\vee \xRightarrow{z_2} \Id_{X^\vee} .
\end{align*}
\end{definition}

\begin{definition}
 Let $\Cc$ be an $(\infty,N)$-category and $f:X\to Y$ and $g:Y \to X$ be $k$-morphisms in $\Cc$ for $k \geq 1$. Then $g$ is a \textit{right adjoint} for $f$ (and $f$ is a \textit{left adjoint} for $g$) if there exists a pair of $(k+1)$-morphisms
    \begin{align*}
    u : \Id_{X} \to g \circ f \quad \text{and} \quad v : f \circ g \to \Id_Y
    \end{align*}
    in $\Cc$, and invertible $(k+2)$-morphisms $z_1$ and $z_2$ exhibiting the zig-zag relations:
    \begin{align*}
    f \xrightarrow{\simeq} f \circ \Id_{X} \xrightarrow{\Id_f \circ u } f\circ g \circ f \xrightarrow{v \circ \Id_f} \Id_Y \circ f \xrightarrow{\simeq} f  \xRightarrow{z_1} \Id_{f}  \\
    g \xrightarrow{\simeq} \Id_X \circ g \xrightarrow{u \circ \Id_{g}} g\circ f \circ g \xrightarrow{\Id_{g} \circ v} g \circ \Id_Y \xrightarrow{\simeq} g  \xRightarrow{z_2} \Id_{g}
    \end{align*}
    The statement ``$g$ is a right adjoint for $f$'' is summarized by the notation $f \dashv g$.
\end{definition}

\begin{remark}\label{rem:coherent}
Equivalently, an object $X$ in $\Cc$ is dualizable if it is dualizable as an object of $h_1\Cc$, the symmetric monoidal homotopy $1$-category of $\Cc$, and a $k$-morphism admits a right adjoint if it does so in an appropriate homotopy bicategory. Both notions may also be formulated homotopy coherently, in terms of functors out of the free adjunction; by the homotopy uniqueness of adjunctions \cite[Theorem~4.4.11]{RV16}, the two formulations agree.
\end{remark}

\begin{definition}
Let $f$ be a $k$-morphism in $\Cc$ for $k\geq 1$. Then $f$ is \textit{right-adjunctible} if it admits a right adjoint and \textit{left-adjunctible} if it admits a left adjoint. We say $f$ is \textit{adjunctible} if $f$ admits a right and left adjoint.
\end{definition}

We will also care about the case where a morphism $f$ is part of an infinite sequence of adjunctions
\begin{align*}
    \cdots  \dashv f^{LL} \dashv f^L \dashv f \dashv f^R \dashv f^{RR} \dashv \cdots .
\end{align*}
In other words, there exists (not necessarily distinct) 1-morphisms $\{f_i\}_{i\in \ZZ}$ in $\Cc$ such that $f_0 = f$ and $f_i \dashv f_{i+1}$ for all $i\in \ZZ$. In this case, we say that $f$ admits a \textit{tower of adjunctions}. We make the following iterative definition.

\begin{definition}\label{def:towers-of-adjunctions}
For $m\geq 1$, a $k$-morphism $f$ in $\Cc$ is \emph{$m$-times $t$-adjunctible} if $f$ admits a tower of adjunctions and, if $m > 1$, all of the units and counits witnessing the adjunctions in the tower are themselves $(m-1)$-times $t$-adjunctible.
\end{definition} 

\begin{definition}
Let $\Cc$ be a symmetric monoidal $(\infty,N)$-category and let $n\geq 1$. Then $\Cc$ \textit{has $n$-duals} if the following conditions hold
\begin{enumerate}
    \item All objects in $\Cc$ admit duals.
    \item For $1\leq k < n$, all $k$-morphisms in $\Cc$ are adjunctible. 
\end{enumerate}
\end{definition}

Following \cite{L09}, let $\Cc^{\mathrm{nd}}$ be the maximal subcategory of $\Cc$ with respect to the property of having $n$-duals; an object $X\in\Cc$ is said to be $n$-dualizable if it belongs to $\Cc^{\mathrm{nd}}$. A definition of this subcategory is sketched in the PhD thesis of Araújo \cite{Ara17}; it admits the following characterization, which is implicit in \cite{L09}. Using the homotopy-coherent definitions of \Cref{rem:coherent}, a definition of $\Cc^{\mathrm{nd}}$ is given in \cite{BV26}, where it is shown to be well-defined and to agree with the same characterization \cite[Corollary~2.15]{BV26}.

\begin{lemma}\label{lem:dual-towers}
For $k\geq 1$, a $k$-morphism $f$ in $\Cc$ belongs to $\Cc^{nd}$ if and only if the source and target of $f$ belong to $\Cc^{nd}$ and $f$ is $(n-k)$-times t-adjunctible.
\end{lemma}

\subsection{Partial dualizability}\label{sec:partial-duals}

We now consider the case where a morphism admits only specified left or right adjoints, and the resulting units and counits again admit specified left or right adjoints, and so forth. This was introduced in \cite{JFS15}, where either right adjoints or left adjoints are required at every level. The notion of iterated adjunctibility studied in \cite{SS23} allows for arbitrary choices of left and right adjoints at successive stages. These choices are conveniently encoded by a dexterity function.

\begin{definition}\label{def:dexterity-fn}
A \textit{dexterity function of length $n$} is a function
\begin{equation*}
    a^n:\{1,...,n\} \to \{L,R\}.
\end{equation*} 
\end{definition}

For a dexterity function $a^n$ of length $n$ and $k< n$, we denote by $a^{n}_{-k}$ the dexterity function of length $(n-k)$ defined by
\begin{align*}
    a^{n}_{-k}(i) := a^n(i+k).
\end{align*}
We will often denote a dexterity function $a^n$ by the corresponding word on the elements $R$ and $L$. For example, $RL$ corresponds to the dexterity function $a^2(1) = R$ and $a^2(2)=L$, while $R^2$ corresponds to the dexterity function $a^2(1) = R$ and $a^2(2)=R$.

\begin{definition}\label{def:n-times-right-left}
Let $a^n$ be a dexterity function of length $n \geq 1$. Let $f$ be a $k$-morphism in $\Cc$. Then we say that $f$ is \emph{$a^n$-adjunctible} if
\begin{equation*}
    \begin{cases} \text{$f$ is right-adjunctible},  &\text{if $a^n(1)$ =R}  \\
     \text{$f$ is left-adjunctible},  &\text{if $a^n(1)$ =L} 
    \end{cases}
\end{equation*}
and if $n>1$, the unit and counit of the adjunction are themselves $a^{n}_{-1}$-adjunctible.
\end{definition}

\begin{remark}
Setting $a^n = R^n$ in Definition \ref{def:n-times-right-left} recovers the notion of $n$-times right-adjunctibility introduced in \cite{JFS15}. Similarly, $a^n = L^n$ in Definition \ref{def:n-times-right-left} recovers the notion of $n$-times left-adjunctibility.
\end{remark}

The Interchange Lemma provides a mechanism for exchanging left and right adjoints. It was stated in \cite[Remark 3.4.22]{L09}, and a proof can be found in \cite[Lemma 1.4.4]{DSPS18}. 

\begin{lemma}[Interchange Lemma]\label{lem:exchange-lem}
Let $\Cc$ be an $(\infty,N)$-category. Let $f$ be a $k$-morphism that admits a left adjoint $f^L$ with unit $\mu$ and counit $\nu$. If $\mu$ and $\nu$ admit left adjoints $\mu^L$ and $\nu^L$, then $(f \dashv f^L, \nu^L,\mu^L)$ is an adjunction. Similarly, if $\mu$ and $\nu$ admit right adjoints $\mu^R$ and $\nu^R$, then $(f \dashv f^L, \nu^R, \mu^R)$ is an adjunction.   
\end{lemma}

The Interchange Lemma is also an essential ingredient in the work of \cite{SS23}, in which Scheimbauer--Stempfhuber extend the exchange of adjoints to the setting of iterated adjunctibility. We refer to these results as the Even-Even Lemma and the Even-Odd Lemma.

\begin{lemma}[{Even-Even Lemma, Scheimbauer--Stempfhuber \cite[Theorem~3.10(1)]{SS23}}]\label{lem:even-even}
Let $\Cc$ be an $(\infty,N)$-category and let $f$ be a $k$-morphism in $\Cc$. Let $a^n$ and $b^n$ be two dexterity functions of length $n$ such that
\begin{align*}
    |(a^n)^{-1}(R)| \equiv |(b^n)^{-1}(R)| \quad \text{mod 2},
\end{align*}then $f$ is $a^n$-adjunctible if and only if it is $b^n$-adjunctible.
\end{lemma}

\begin{example}
Suppose that $n=2$, then Lemma \ref{lem:even-even} implies that a 1-morphism $f$ is 2-times right-adjunctible if and only if it is 2-times left-adjunctible. Hence, 2-times left-adjunctible and 2-times right-adjunctible 1-morphisms coincide. However, there is a second notion of twice partially adjunctible 1-morphisms, namely $LR$-adjunctible (or $RL$-adjunctible), which does not coincide with 2-times left/right-adjunctible.
\end{example}

\begin{lemma}[{Even-Odd Lemma,  Scheimbauer--Stempfhuber \cite[Theorem~3.10(2)]{SS23}}]\label{lem:even-odd}
Let $\Cc$ be an $(\infty,n)$-category and let $f$ be a $k$-morphism in $\Cc$. Let $a^n$ and $b^n$ be two dexterity functions of length $n$ such that
\begin{align*}
    |(a^n)^{-1}(R)| + 1  \equiv |(b^n)^{-1}(R)| \quad \text{mod 2},
\end{align*}
then $f$ is $n$-times adjunctible if and only if $f$ is $a^n$-adjunctible and $b^n$-adjunctible.
\end{lemma} 

\begin{example}
Suppose that $n=2$, then Lemma \ref{lem:even-odd} implies that if a 1-morphism $f$ is $RR$-adjunctible as well as $LR$-adjunctible then it is automatically 2-times adjunctible. More generally, let $w^{(n-1)}$ be a dexterity function of length $(n-1)$. If a 1-morphism $f$ is $Rw^{(n-1)}$-adjunctible and $Lw^{(n-1)}$-adjunctible, then it is automatically $n$-times adjunctible.
\end{example}

\subsection{Sufficient conditions for $n$-dualizability}\label{sec:Araujo}

In \cite{L09}, Lurie provides necessary and sufficient conditions for an object $X\in \Cc$ to be 2-dualizable.

\begin{lemma}[{Lurie \cite[Proposition 4.2.3]{L09}}]\label{lem:Lur-2-dual}
Let $\Cc$ be a symmetric monoidal $(\infty,N)$-category and $X\in \Cc$ be an object. Then $X$ is 2-dualizable if and only if the following conditions are satisfied:
\begin{enumerate}
    \item The object $X$ is 1-dualizable.
    \item The evaluation map $\ev_X$ admits a right and left adjoint.
\end{enumerate}
\end{lemma}

\begin{remark}
In Lemma \ref{lem:Lur-2-dual}, it is implicit that condition (2) is independent of the choice of evaluation map $\ev_X$.
\end{remark}

Necessary and sufficient conditions for an object $X\in \Cc$ to be $n$-dualizable are provided in the PhD thesis of Araújo \cite{Ara17}, which we now review. The following result is part of the proof of \cite[Theorem 4.1.19]{Ara17}. For clarity, we provide a proof. 

\begin{lemma}\label{lem:adj-2-fully-adj}
Let $f$ be a $k$-morphism in $\Cc$ and let $n > k+ 1$. Then $f$ is $(n-k)$-times t-adjunctible if and only if $f$ is $(n-k)$-times adjunctible.
\end{lemma}

\begin{proof}
The forward direction is immediate, so we just need to prove the reverse. Let $f$ be a $k$-morphism in $\Cc$ that is $(n-k)$-times adjunctible. Observe that $f$ is 2-times adjunctible since $n-k > 1$. By the Interchange Lemma, a right adjoint of $f$ is also a left adjoint of $f$ and vice versa. So $f$ admits a tower of adjunctions. These observations imply that $f$ is $(n-k-1)$-times t-adjunctible. It remains to show that the units and counits at the top level also admit towers of adjunctions.

Suppose that $u$ and $v$ are the unit and counit for an adjunction $g \dashv g^R$ and that $u$  and $v$ both admit left and right adjoints. By the Interchange Lemma, $(g^R,v^R,u^R)$ and $(g^R,v^L,u^L)$ both provide left adjoint data for $g$. By the standard uniqueness of adjunction data, we have that $v^R$ is equivalent to $v^L$ up to whiskering with invertible morphisms. Similarly, $u^R$ is equivalent to $u^L$ up to whiskering with invertible morphisms. Hence $v^{R}$ and $u^{R}$ admit right adjoints and $v^{L}$ and $u^{L}$ admit left adjoints. This observation can be inductively applied to produce a tower of adjunctions for $u$ and $v$.
\end{proof}

\begin{remark}
This is not true in the case that $n-k=1$. For example, a 1-morphism in a bicategory can have both a left and right adjoint, but not admit a tower of adjunctions.
\end{remark}

The following result of Araújo gives a reduction of the duality data required to show that $X\in \Cc$ is $n$-dualizable. For a proof in the language of dexterity functions, see \cite[Proposition B.5]{SS23}.

\begin{theorem}[{Araújo \cite[Theorem 4.1.19]{Ara17}}]\label{thm:Araujo}
Let $\Cc$ be a symmetric monoidal $(\infty,N)$-category. Then an object $X$ is $n$-dualizable if and only if $X$ is dualizable and one of the following equivalent conditions holds
\begin{enumerate}
    \item $\ev$ and $\coev$ are $(n-1)$-times t-adjunctible.
    \item $\ev$ and $\coev$ are $(n-1)$-times adjunctible.
    \item $\ev$ and $\coev$ are $a^{(n-1)}$-adjunctible for some dexterity function $a^{(n-1)}$ of length $(n-1)$.
\end{enumerate}
\end{theorem}

\begin{remark}
In \cite{JFS15}, an equivalence between 2 and 3 is provided in the case that $a^{(n-1)}$ is either $R^{(n-1)}$ or $L^{(n-1)}$.
\end{remark}

\subsection{Sufficient conditions for ambient $n$-dualizability}\label{sec:def-dual-cond}

This section is devoted to proving Theorem \ref{thm:main-dual-result}. Let $\Cc$ be a symmetric monoidal $(\infty,N)$-category, let $1 \leq k < n$, and let $a^{(n-k)}$ be a dexterity function of length $n-k$.

\ThmMainDual*

By definition, there is a series of implications
\begin{align}
    f \text{ is ambiently $n$-dualizable}
    &\Rightarrow f \text{ is $(n-k)$-times adjunctible} \tag{I-1}\label{eq:implication-1} \\
    &\Rightarrow f \text{ is $a^{(n-k)}$-adjunctible}. \tag{I-2}\label{eq:implication-2}
\end{align}
Theorem \ref{thm:main-dual-result} asserts that, when the source and target of $f$ are ambiently $n$-dualizable, both implications may be reversed.

The proof proceeds in two stages. We first establish Theorem \ref{thm:main-dual-result} for 1-morphisms between $n$-dualizable objects. The key step is a Left-Right Lemma, which states that left and right adjoints are related by Serre automorphisms. We then introduce Folding Lemmas that reduce the general higher-categorical case to the 1-morphism case.

\subsubsection{The bicategory case}

Let $\Cc$ be a symmetric monoidal bicategory and denote its braiding by $\sigma$. Let $X$ be a 2-dualizable object in $\Cc$. Let $(X,X^\vee,\ev_X,\coev_X)$ be duality data in $\Cc$ and let $\ev_X^R$ be a right adjoint for $\ev_X$. Denote the composition $\sigma_{X^\vee,X} \circ \ev_X^R$ by $\widetilde{\ev}_X^R$. The \textit{Serre automorphism of $X$} is defined by
\begin{align*}
    S_X :=  (\widetilde{\ev}_X^R \otimes \Id_{X}) \circ (\Id_{X} \otimes \ev_X) \, ,
\end{align*}
As observed in \cite[Proposition~4.2.3]{L09} and proven rigorously by Pstr{\k{a}}gowski \cite{Pst22}, the Serre automorphism is invertible with inverse given by
\begin{align*}
    S_X^{-1} =  (\widetilde{\ev}_X^L \otimes \Id_{X}) \circ (\Id_{X} \otimes \ev_X) .
\end{align*}
Here, $\widetilde{\ev}^L_X = \sigma_{X^\vee,X}\circ \ev_X^L$, where $\ev^L_X$ is a left adjoint of $\ev_X$. The following lemma shows that the Serre automorphisms of $X$ and $Y$ identify left and right adjoints.

\begin{lemma}[Left-Right Lemma]\label{lem:left-right-lem}
Let $\Cc$ be a symmetric monoidal bicategory. Let $X$ and $Y$ be 2-dualizable objects in $\Cc$. Fix duality data $(X,X^\vee,\ev_X,\coev_X)$, a choice of right and left adjoint for $\ev_X$, and similarly for $Y$. Let $f:X\to Y$ be a 1-morphism in $\Cc$. Then $g$ is a right adjoint to $f$ if and only if
\begin{align}
    S_X^{-1} \circ g \circ S_Y 
\end{align}
is a left adjoint for $f$.
\end{lemma}

\begin{remark}
In the case $f$ is either $\ev_X$ or $\coev_X$, an analogous result is provided in \cite[Proposition 4.2.3]{L09}, a rigorous proof is given in \cite[Theorem 3.9]{Pst22}.
\end{remark}

Before we prove Lemma \ref{lem:left-right-lem}, consider the following motivating example.

\begin{example}\label{ex:left-right-bord}
The example that motivates Lemma \ref{lem:left-right-lem} is the case that $\Cc$ is the 2-dimensional framed domain wall bordism bicategory. Let $A$ and $B$ denote the two regions separated by the domain wall. Take $X = \pt_A^+$ and $Y = \pt_B^+$ to be the positively framed points labeled by regions $A$ and $B$ respectively. The Serre automorphism for $X$ and $Y$ are given by appropriately labeled intervals in which the 2-framing undergoes a single rotation. For example, the Serre automorphism for $X$ and it's inverse are depicted as follows:

{\vspace{6mm}\centering\includegraphics[width=0.9\textwidth]{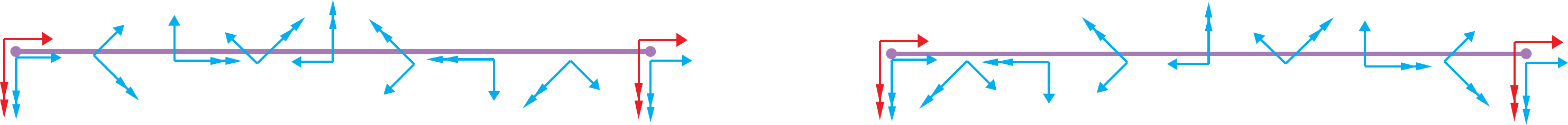}\par\vspace{5mm}}

\noindent We use purple to indicate a labeling by $A$ and yellow a labeling by $B$, and green to indicate the interface. The blue arrows depict the framing vectors. The red arrows indicate the framing of the inflated tangent bundle; these are called arrows of time in \cite{FT21}. The bordisms are read in the direction of the red arrows, left to right for 1-morphisms and down the page for 2-morphisms. Let $f$ be the standard interval that crosses the domain wall from $A$ to $B$:

{\vspace{6mm}\centering\includegraphics[width=0.45\textwidth]{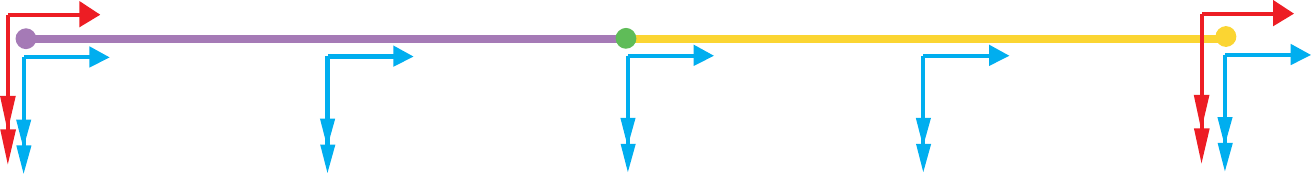}\par\vspace{5mm}}

\noindent Note that the framing at the codimension one defect must be such that the first component is always inward normal to the yellow region (the choice of the first component is a convention; one could equally choose the second component). The right adjoint for $f$ is given by

{\vspace{6mm}\centering\includegraphics[width=0.45\textwidth]{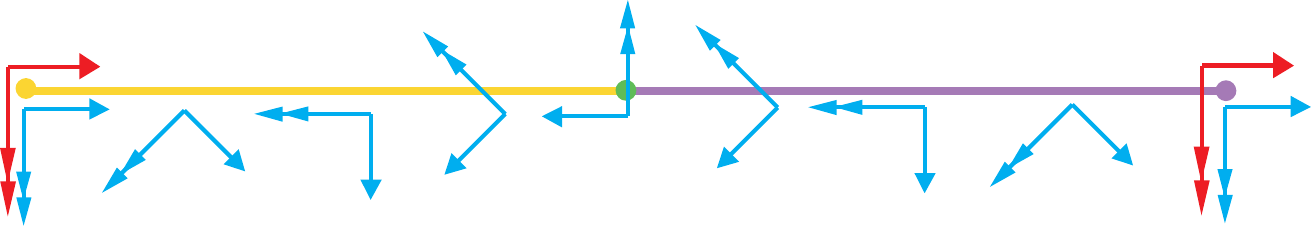}\par\vspace{5mm}}

\noindent Here the framing rotates by a half turn as it approaches the domain wall from either side. The unit and counit are given by the following 2-bordisms:

{\vspace{6mm}\centering\includegraphics[width=0.9\textwidth]{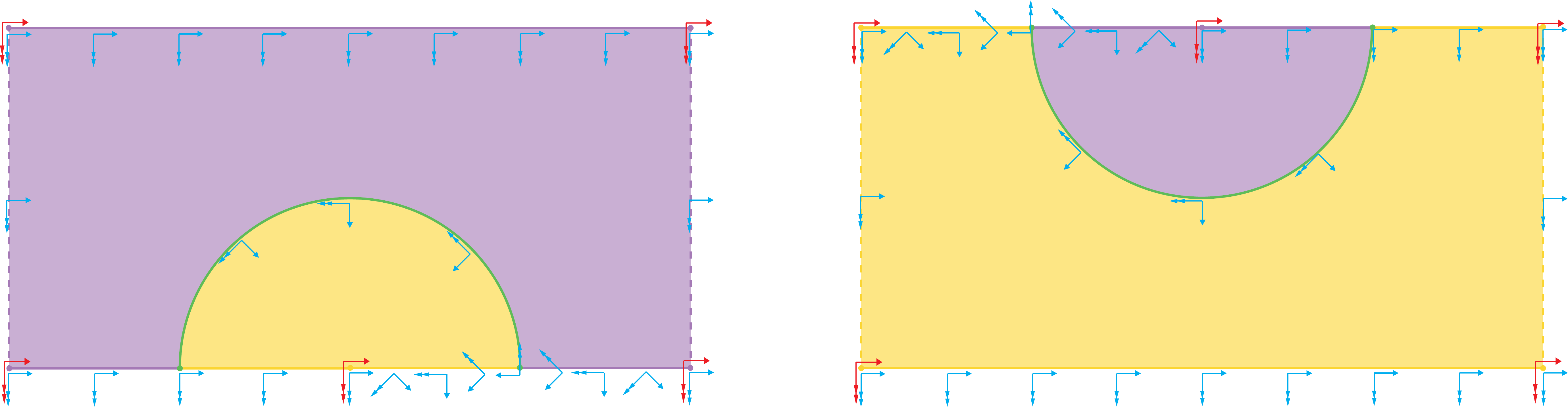}\par\vspace{5mm}}

\noindent Imlpicit in these diagrams is the fact that the 2-framing can be extended to the entire manifold with corners. In contrast, the left adjoint for $f$ is given by

{\vspace{6mm}\centering\includegraphics[width=0.45\textwidth]{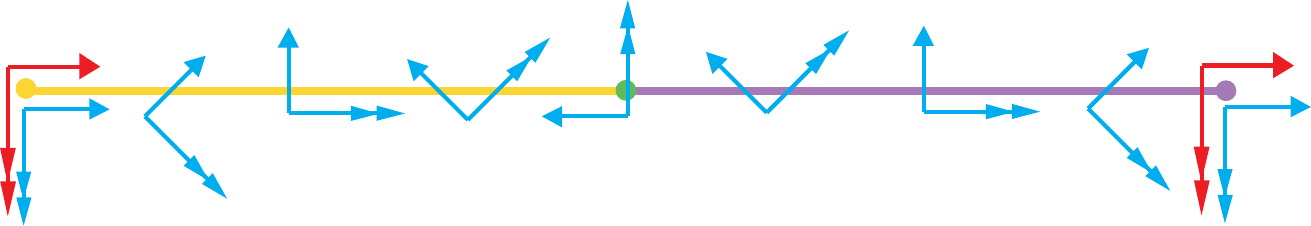}\par\vspace{5mm}}

\noindent The framing again rotates by a half a turn as it approaches the domain wall, but now in the opposite direction to that of the right adjoint. Since the Serre automorphisms provide a complete rotation of the framing, we have that the left and right adjoints are identified by composing with Serre automorphisms:

{\vspace{6mm}\centering\includegraphics[width=0.9\textwidth]{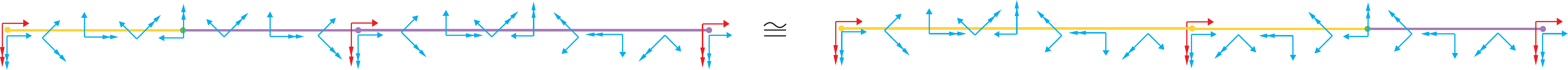}\par\vspace{5mm}}

\end{example}

We will prove the Left-Right Lemma shortly. First, we review some results about adjunctions in symmetric monoidal bicategories.

\begin{lemma}\label{lem:adjoint-of-product}
Let $\Cc$ be a symmetric monoidal bicategory. Let $f$ and $g$ be 1-morphisms in $\Cc$. Let $f^R$ and $g^R$ be right adjoints for $f$ and $g$, respectively. Then $f^R \otimes g^R$ is a right adjoint for $f \otimes g$.

Moreover, the unit (counit) of the adjunction $f \otimes g \dashv f^R \otimes g^R$ is given by the tensor product of the units (counits) of the adjunction $f \dashv f^R$ and $g \dashv g^R$ up to composition with isomorphisms.
\end{lemma}

\begin{proof}
Let $(f\dashv  f^R, u_f^R,v_f^R)$ and $(g\dashv g^R,u_g^R,v_g^R)$ be adjunction data in $\Cc$. Then $(f \times g \dashv f^R \times g^R,u_f^R\times u_g^R,v_f^R \times v_g^R)$ is adjunction data in $\Cc \times \Cc$. Since $\otimes : \Cc \times \Cc \to \Cc$ is a functor of bicategories $f^R \otimes g^R$ is a right adjoint of $f \otimes g$.

Moreover, the unit (counit) of the adjunction is given by the image of $u_f^R\times u_g^R$ ($v_f^R \times v_g^R$) under the functor $\otimes$, composed with structural isomorphisms for $\otimes$.
\end{proof}

Let $f:X\to Y$ be a 1-morphism in $\Cc$ and fix choices $(X,X^\vee,\ev_X,\coev_X)$ and $(Y,Y^\vee,\ev_Y,\coev_Y)$ of duality data in $\Cc$. The \textit{dual} $f^\vee:X^\vee \to Y^\vee$ for $f$ is given by
\begin{equation*}
\begin{tikzcd}[column sep=small]
	{Y^\vee} && {Y^\vee \otimes X \otimes X^\vee} && {Y^\vee \otimes Y \otimes X^\vee} && {X^\vee}
	\arrow["{\Id \otimes \coev_X}", from=1-1, to=1-3]
	\arrow["{\Id \otimes f \otimes \Id}", from=1-3, to=1-5]
	\arrow["{\ev_Y \otimes\Id}", from=1-5, to=1-7]
\end{tikzcd}
\end{equation*}
The isomorphism class of $f^\vee$ does not depend on the choices of duality data. Moreover, $(f^\vee)^\vee \cong f$.

\begin{lemma}\label{lem:check-preserves-adjoints}
Let $\Cc$ be a symmetric monoidal bicategory. Let $(X,X^\vee,\ev_X,\coev_X)$ and $(Y,Y^\vee,\ev_Y,\coev_Y)$ be duality data in $\Cc$. Let $f:X\to Y$ be a 1-morphism in $\Cc$ with right adjoint $f^R$, then $f^\vee$ is a right adjoint for $(f^R)^\vee$.

Moreover, the unit (counit) of the adjunction $(f^R)^\vee \dashv f^\vee$ is given by the unit (counit) of the adjunction $f \dashv f^R$ up to tensor product and composition with isomorphisms.
\end{lemma}

\begin{remark}
Conceptually, Lemma \ref{lem:check-preserves-adjoints} expresses the familiar idea that passing to duals is a contravariant functor and therefore exchanges left and right adjoints. The proof below formalizes this observation using the oplax arrow category and \cite[Proposition~7.10]{JFS15}.
\end{remark}

\begin{proof}
By \cite[Proposition 7.10]{JFS15}, a 1-morphism $g:X^\vee \to Y^\vee$ in $\Cc$ is a dual to $f$ as objects in the oplax arrow category $\Cc^{\to}$ if and only if $g^\vee \simeq f^R$. Hence, $(f^R)^\vee$ is a dual to $f$ in $\Cc^\to$. Since $\Cc^{\to}$ is symmetric monoidal, $f$ is also a dual to $(f^R)^\vee$. It follows that $f^\vee$ is a right adjoint for $(f^R)^\vee$.

Moreover, the evaluation 1-morphism for $(f^R)^\vee$ as a dual to $f$ in $\Cc^\to$ is a diagram in $\Cc$ of the form
\begin{equation*}
\begin{tikzcd}
	{X \otimes X^\vee} && {Y \otimes Y^\vee} \\
	\\
	\unit && \unit
	\arrow["{f \otimes (f^R)^\vee}", from=1-1, to=1-3]
	\arrow["{\ev_X}"', from=1-1, to=3-1]
	\arrow["v", shorten <=12pt, shorten >=12pt, Rightarrow, from=1-3, to=3-1]
	\arrow["{\ev_Y}", from=1-3, to=3-3]
	\arrow["\Id"', from=3-1, to=3-3]
\end{tikzcd}
\end{equation*}
where $v$ is given by the counit of the adjunction $f\dashv f^R$ up to composition and tensor product with isomorphisms. Similarly, the coevaluation 1-morphism for $(f^R)^\vee$ as a dual to $f$ in $\Cc^\to$ is given by a diagram in $\Cc$ of the form
\begin{equation*}
    \begin{tikzcd}
	\unit && \unit \\
	\\
	{X \otimes X^\vee} && {Y \otimes Y^\vee}
	\arrow["\Id", from=1-1, to=1-3]
	\arrow["{\coev_X}"', from=1-1, to=3-1]
	\arrow["u", shorten <=12pt, shorten >=12pt, Rightarrow, from=1-3, to=3-1]
	\arrow["{\coev_Y}", from=1-3, to=3-3]
	\arrow["{f \otimes (f^R)^\vee}"', from=3-1, to=3-3]
\end{tikzcd}
\end{equation*}
where $u$ is given by the unit of the adjunction $f\dashv f^R$ up to composition and tensor product with isomorphisms. Since, up to composition with isomorphisms, the same 1-morphisms in $\Cc^\to$ describe $f$ as a dual to $(f^R)^\vee$, the unit (counit) of the adjunction $(f^R)^\vee \dashv f^\vee$ is given by the unit (counit) of the adjunction $f \dashv f^R$ up to tensor product and composition with isomorphisms.
\end{proof}

We will also use other standard results regarding left and right adjoints in bicategories. In particular, the inverse of an invertible 1-morphism is both a left and right adjoint for that 1-morphism, and the composition of two right adjoints is a right adjoint for the composition.

\begin{proof}[Proof of Lemma \ref{lem:left-right-lem}]
Suppose that $f \dashv f^R$. By Lemma \ref{lem:check-preserves-adjoints} we have that $(f^R)^\vee \dashv f^\vee$. Since identities are their own adjoints, Lemma \ref{lem:adjoint-of-product}  gives $(f^R)^\vee \otimes \Id_Y \dashv f^\vee \otimes \Id_Y$. Since $\ev_Y^R$ is a right adjoint of $\ev_Y$, and right adjoints compose, we have
\begin{align}\label{eq:first-adj-rel}
    \ev_Y \circ ((f^R)^\vee \otimes \Id_{Y}) \dashv (f^\vee \otimes \Id_{Y})\circ \ev_Y^R
\end{align}
In terms of string diagrams, relation \eqref{eq:first-adj-rel} is depicted by
\begin{equation*}
   \begin{tikzpicture}[scale=0.6]
        \draw (-2,-5.1) -- (-2,0);
        \draw (0,0) arc (0:180:1);
        \node[anchor=south] at (-1,1) (A) {$\ev_X$};
        \draw (0,0) -- (0,-0.5);
        \draw (-0.5,-0.5) rectangle (0.5,-1.5);
        \draw(0,-1.5) -- (0,-2);
        \node at (0,-1) {$f^R$};
        \draw (2,-2) -- (2,0);
        \draw (0,-2) arc (180:360:1);
        \node[anchor=north] at (1,-3) (B) {$\coev_Y$};
        \draw (2,0) -- (2,2);
        \draw (2,2) arc (180:0:1);
        \node[anchor=south] at (3,3) (C) {$\ev_Y$};
        \draw (4,-5.1) -- (4,2);
        \draw[-<] (-2,-5.5) -- (-2,-5);
        \node[anchor=south east] at (-2,-5.5) {$X$};
        \draw[->] (4,-5.5) -- (4,-5);
        \node[anchor=south east] at (4,-5.5) {$Y$};

        \draw (5.7,-1) -- (6.3,-1);
        \draw (6.3,-0.8) -- (6.3,-1.2);

        \draw (8,-2) -- (8,2);
        \draw (10,2) arc (0:180:1);
        \node[anchor=south] at (9,3) (A) {$\text{ev}_Y$};
        \draw (10,2) -- (10,1.5);
        \draw (9.5,1.5) rectangle (10.5,0.5);
        \draw(10,0.5) -- (10,0);
        \node at (10,1) {$f$};
        \draw (12,0) -- (12,2);
        \draw (10,0) arc (180:360:1);
        \node[anchor=north] at (11,-1) (B) {$\text{coev}_X$};
        \draw (12,2) -- (12,4.1);
        \draw[<-] (12,4) -- (12,4.5);
        \node[anchor=north east] at (12,4.5) {$X$};
        \draw (8,-2) arc (180:360:2.5);
        \node[anchor=north] at (10.5,-4.5) (C) {$\ev^R_Y$};
        \draw (13,-2) -- (13,4.1);
        \node[anchor=north west] at (13,4.5) {$Y$};
        \draw[>-] (13,4) -- (13,4.5);
    \end{tikzpicture}
\end{equation*}
Composing with the adjunction $\ev_X^L \dashv \ev_X$, and applying Lemma \ref{lem:adjoint-of-product} and compatibility of right adjoints with composition, gives
\begin{equation}\label{eq:adjoints-for-left-right-lem}
    (\Id_X \otimes (\ev_Y \circ ((f^R)^\vee \otimes \Id_Y))) \circ (\widetilde{\ev}_X^L \otimes \Id_{Y}) \dashv (\widetilde{\ev}_X \otimes \Id_Y) \circ (((\Id_{Y} \otimes f^\vee) \circ \ev_Y^R) \otimes \Id_Y)
\end{equation}
which in terms of string diagrams is depicted by
\begin{equation*}
    \begin{tikzpicture}[scale=0.5]
        %%%% First picture %%%
        \draw (-2,-4) -- (-2,0);
        \draw (0,0) arc (0:180:1);
        \node[anchor=south] at (-1,1) (A) {$\ev_X$};
        \draw (0,0) -- (0,-0.5);
        \draw (-0.5,-0.5) rectangle (0.5,-1.5);
        \draw(0,-1.5) -- (0,-2);
        \node at (0,-1) {$f^R$};
        \draw (2,-2) -- (2,0);
        \draw (0,-2) arc (180:360:1);
        \node[anchor=north] at (1,-3) (B) {$\coev_Y$};
        \draw (2,0) -- (2,2);
        \draw (2,2) arc (180:0:1);
        \node[anchor=south] at (3,3) (C) {$\ev_Y$};
        \draw (4,-8) -- (4,2);
        \draw[->] (4,-8.5) -- (4,-8);
        \node[anchor=south east] at (4,-8.5) {$Y$};

        \cross(-4,-4)(-2,-6);
        
        \draw (-4,-6) arc (180:360:1);
        \node[anchor=north] at (-3,-7) (D) {$\text{ev}^L_X$};
        
        \draw (-4,-4) -- (-4,8.1);
        \draw[>-] (-4,8) -- (-4,8.5);
        \node[anchor=north west] at (-4,8.5) {$X$};

        %%% adjoint symbol %%%
        \draw (6.2,-1) -- (6.8,-1);
        \draw (6.8,-0.8) -- (6.8,-1.2);

        %% second picture %%%
        \draw (10,-4) -- (10,0);
        \draw (12,0) arc (0:180:1);
        \node[anchor=south] at (11,1) (A) {$\text{ev}_Y$};
        \draw (12,0) -- (12,-0.5);
        \draw (11.5,-0.5) rectangle (12.5,-1.5);
        \draw(12,-1.5) -- (12,-2);
        \node at (12,-1) {$f$};
        \draw (14,-2) -- (14,0);
        \draw (12,-2) arc (180:360:1);
        \node[anchor=north] at (13,-3) (B) {$\text{coev}_X$};
        \draw (14,0) -- (14,2);
        
        \draw (10,-4) arc (180:360:2.5);
        \node[anchor=north] at (12.5,-6.5) (C) {$\ev^R_Y$};
        
        \draw (15,-4) -- (15,8.1);
        \node[anchor=north west] at (15,8.5) {$Y$};
        \draw[>-] (15,8) -- (15,8.5);

        \cross(9,4)(14,2);
        
        \draw (9,4) arc (180:0:2.5);
        \node[anchor=south] at (11.5,6.5) (D) {$\ev_X$};
        
        \draw (9,-8) -- (9,2);
        \draw[<-] (9,-8) -- (9,-8.5);
        \node[anchor=south east] at (9,-8.5) {$X$};
    \end{tikzpicture}
\end{equation*}
It remains to identify the two sides of \eqref{eq:adjoints-for-left-right-lem} with expressions involving the Serre automorphisms. By expanding the definition of $(f^R)^\vee$ and applying a cusp isomorphism for $Y$, the left hand side of \eqref{eq:adjoints-for-left-right-lem} is seen to be isomorphic to
\begin{equation*}
   (\Id_X \otimes \ev_X) \circ (\widetilde{\ev}_X^L \otimes f^R)  \cong S_X^{-1} \circ f^R .
\end{equation*}
Similarly, by expanding the definition of $f^\vee$ and applying a cusp isomorphism for $X$, the right hand side of \eqref{eq:adjoints-for-left-right-lem} is seen to be isomorphic to
\begin{equation*}
     (\Id_{Y} \otimes \ev_X) \circ (\widetilde{\ev}_{Y}^R \otimes f) \cong  S_Y \circ f .
\end{equation*}
It follows that $S_X^{-1} \circ f^R \circ S_Y$ is a left adjoint for $f$. The converse follows by applying the same arguments to the adjunction $S_X^{-1} \circ f^R \circ S_Y \dashv f$.
\end{proof}

In proving the Left-Right Lemma, we did not specify the data of a unit and counit of adjunction; we just showed that they exist. However, Lemma \ref{lem:adjoint-of-product} and \ref{lem:check-preserves-adjoints} both provide a description of the unit and counit up to compositions and tensor products with isomorphisms. Unwinding our arguments, we see that the unit for the adjunction $S_X^{-1} \circ f^R \circ S_Y \dashv f$ is given by
\begin{align}
\begin{split}\label{eq:left-right-unit}
    \Id_{Y} &\xrightarrow{``u^L_{ev_X}"} (\ev_X \circ \ev_X^L) \otimes \Id_Y  \\ 
    &\xrightarrow{``u^R_f"} (\ev_X \circ (\Id_{X^\vee} \otimes (f^R \circ f)) \circ \ev_X^L) \otimes \Id_Y \\
    &\xrightarrow{``u^R_{\ev_Y}"} f \circ S_X^{-1} \circ  f^R \circ S_Y .
\end{split}
\end{align} 
The notation $``u^R_f"$ indicates that the 2-moprhism is given by $u^R_f$ up to compositions and tensor products with isomorphisms. Similarly, the counit is given by
\begin{align}
\begin{split}\label{eq:left-right-counit}
    S_X^{-1} \circ f^R \circ S_Y \circ f &\xrightarrow{``v_{\ev_X}^L"} (\ev_Y \circ (\Id_{Y^\vee} \otimes (f\circ f^R)) \circ \ev_Y^R) \otimes \Id_X \\
    &\xrightarrow{``v_f^R"} (\ev_Y^R \circ \ev_Y) \otimes \Id_X \\
    &\xrightarrow{``v^R_{\ev_Y}"} \Id_{X}.
\end{split}
\end{align}
These observations will be important when extending the Left-Right Lemma to higher categories.

\begin{corollary}\label{cor:left-right-cor}
Let $\Cc$ be a symmetric monoidal bicategory. Let $f:X\to Y$ be a 1-morphism and suppose that $X$ and $Y$ are 2-dualizable. Then the following are equivalent
\begin{enumerate}[label=\arabic*)]
    \item $f$ is ambiently 2-dualizable.
    \item $f$ is adjunctible.
    \item $f$ admits a right adjoint.
    \item $f$ admits a left adjoint.
\end{enumerate}
\end{corollary}

\begin{proof}
The implications $(1) \Rightarrow (2)$ and $(2) \Rightarrow (4)$ are immediate from the definitions. The Left-Right Lemma gives the implication $(4) \Rightarrow  (3)$. To see that $(3) \Rightarrow (1)$, repeatedly apply the Left-Right Lemma to produce a tower of adjunctions
\begin{align*}
    f_i = \begin{cases} (S_X)^{i/2} \circ f \circ (S_Y^{-1})^{i/2} & i \text{ even}\\
    (S_X)^{(i-1)/2} \circ f^{R} \circ (S_Y^{-1})^{
    (i-1)/2} & i \text{ odd}.
     \end{cases}
\end{align*}
By Lemma \ref{lem:dual-towers}, $f$ is ambiently 2-dualizable.
\end{proof}

\subsubsection{Higher ambient dualizability for 1-morphisms}

Let $\Cc$ be a symmetric monoidal $(\infty,N)$-category. We now prove Theorem \ref{thm:main-dual-result} in the case of a 1-morphism $f:X\to Y$ such that $X$ and $Y$ are $n$-dualizable objects in $\Cc$. First, we show the converse to implication \eqref{eq:implication-2}.

\begin{lemma}[Higher Left-Right Lemma]\label{lem:higher-left-right-lem}
Let $\Cc$ be a symmetric monoidal $(\infty,N)$-category and let $X$ and $Y$ be $n$-dualizable objects in $\Cc$. Let $a^{(n-1)}$ be a dexterity function of length $n-1$. Let $f:X \to Y$ be a 1-morphism in $\Cc$. The following are equivalent
\begin{enumerate}[label=(\arabic*)]
    \item $f$ is $(n-1)$-times adjunctible.
    \item $f$ is $a^{(n-1)}$-adjunctible.
\end{enumerate}
\end{lemma}

\begin{remark}
Taking the dexterity function $a^{(n-1)} = R^{n-1}$ (resp. $L^{n-1}$), condition (2) is that $f$ is $(n-1)$-times right (resp. left) adjunctible.
\end{remark}

\begin{proof}
The implication $(1) \Rightarrow (2)$ is immediate from the definitions. We must show the implication $(2) \Rightarrow (1)$. Let $w=a^{(n-1)}_{-1}$ denote the truncated dexterity function obtained by deleting the first entry of $a^{(n-1)}$, and assume without loss of generality that $a^{(n-1)}(1)=R$. Then $f$ admits a right adjoint $f^R$. By the Left-Right Lemma \ref{lem:left-right-lem} we have that $f$ also admits a left adjoint $f^L$. Note that the unit and counit of the adjunction $f \dashv f^R$ are $w$-adjunctible by Definition \ref{def:n-times-right-left}. We claim that the unit and counit of the adjunction $f^L \dashv f$ are also $w$ adjunctible.

As shown in \eqref{eq:left-right-unit}, the unit of the adjunction $f^L \dashv f$ is given by taking compositions and tensor products of isomorphisms, along with the unit for the adjunction $f\dashv f^R$, the unit for the adjunction $\ev_X^L \dashv \ev_X$, and the unit for the adjunction $\ev_Y \dashv \ev_Y^R$. Since $X$ and $Y$ are $n$-dualizable by assumption, their $n$-dualizability data belongs to the maximal $n$-dualizable subcategory $\Cc^{nd}$ of $\Cc$. In particular, the units for the adjunctions $\ev_X^L \dashv \ev_X$ and $\ev_Y \dashv \ev_Y^R$ belong to $\Cc^{nd}$ and therefore $w$-adjunctible. Since $f$ is $a^{(n-1)}$-adjunctible, the unit $u_f^R$ is $w$-adjunctible by definition.  Hence, the unit for the adjunction $f^L \dashv f$ is given by taking compositions and tensor products of $w$-adjunctible morphisms, and is therefore $w$-adjunctible. By the same reasoning, the counit of the adjunction $f^L \dashv f$ is $w$-adjunctible.

Let $Rw$ (resp. $Lw$) denote the dexterity function obtained by prepending $R$ (resp. $L$) to $w$. By definition $f$ is $a^{(n-1)} = Rw$-adjunctible. We have shown that the unit and counit of the adjunction $f^L \dashv f$ are $w$-adjunctible.  So $f$ is also $Lw$-adjunctible. Clearly $|(Rw)^{-1}(R)| +1 \equiv |(Lw)^{-1}(R)| \mod{2}$, so by the Even-Odd Lemma of \cite{SS23}, we have that $f$ is $(n-1)$-times adjunctible.
\end{proof}

It remains to show that implication \eqref{eq:implication-1} can be reversed for 1-morphisms whose source and target are $n$-dualizable.

\begin{lemma}\label{lem:duality-induct-up}
Let $\Cc$ be a symmetric monoidal $(\infty,N)$-category and $X$ and $Y$ be $n$-dualizable objects in $\Cc$. Then a 1-morphism $f:X\to Y$ in $\Cc$ is ambiently $n$-dualizable if and only if $f$ is $(n-1)$-times adjunctible . 
\end{lemma}

\begin{proof}
The forward direction is immediate. For $n=1$, the reverse direction holds since all 1-morphisms $f:X\to Y$ are ambiently 1-dualizable. For $n=2$, Corollary \ref{cor:left-right-cor} implies that $f$ admits a tower of adjunctions and is therefore ambiently 2-dualizable. The case that $n\geq 3$ is a general result given in Lemma \ref{lem:adj-2-fully-adj}, following arguments of \cite{Ara17}.
\end{proof}

\subsubsection{Higher ambient dualizability for k-morphisms}

To complete the proof of Theorem \ref{thm:main-dual-result}, we need to prove that the reverse implications of \eqref{eq:implication-1} and  \eqref{eq:implication-2} also hold for $k>1$. The strategy is to reduce statements about $k$-morphisms to statements about $(k-1)$-morphisms. Folding Lemma~I uses duality of objects to identify morphisms $X\to Y$ with morphisms $\unit\to Y\otimes X^\vee$. Folding Lemma~II then uses adjunction data to identify $2$-morphisms between such morphisms with 1-morphisms in the looping $\Omega\Cc := \Hom_{\Cc}(\unit,\unit)$. The Folding Lemmas state that, given sufficient dualizability of the source and target, a morphism enjoys the same dualizability as its mate. Similar results are also used in \cite{Hai23}.

Let $\Cc$ be a symmetric monoidal $(\infty,N)$-category and let $n\geq 2$. Let $X$ and $Y$ be $n$-dualizable objects in $\Cc$. Let $(X,X^\vee,\ev_X,\coev_X)$ be duality data for $X$. Then, $\coev_X$ induces a functor
\begin{align*}
    \Phi_1: \Hom_{\Cc}(X,Y) &\xrightarrow{\simeq} \Hom_{\Cc}(\unit,Y\otimes X^\vee)   \label{eq:folding-functor-1} \\
    c &\mapsto (c \otimes \Id_{X^\vee}) \circ \coev_X \ .
\end{align*}

\begin{lemma}[Folding Lemma I]\label{lem:gen-fold-1}
The functor $\Phi_1$ is an equivalence. Moreover, a $k$-morphism $f$ in $\Cc$ that belongs to $\Hom_{\Cc}(X,Y)$ is ambiently $n$-dualizable (resp. $a^{(n-k)}$-adjunctible) in $\Cc$ if and only if $\Phi_1(f)$ is ambiently $n$-dualizable (resp. $a^{(n-k)}$-adjunctible) in $\Cc$.
\end{lemma}

\begin{proof}
An inverse to $\Phi_1$ is given by
\begin{align*}
    \Phi_1^{-1}:\Hom_{\Cc}(\unit,Y\otimes X^\vee)   &\to \Hom_{\Cc}(X,Y),\\
                   c' &\mapsto   (\Id_Y \otimes \ev_X) \circ (c' \otimes \Id_X)  ,
\end{align*}
where we omit structural isomorphisms. The fact that this is an inverse follows from the duality relations for $X$. Let $f$ be an ambiently $n$-dualizable (resp. $a^{(n-k)}$-adjunctible) $k$-morphism in $\Cc$ that belongs to $\Hom_{\Cc}(X,Y)$. Since $X$ is $n$-dualizable the morphisms $\coev_X$ and $\ev_X$ belong to $\Cc^{nd}$ (i.e. are ambiently $n$-dualizable 1-morphims), and are therefore $a^{(n-k)}$-adjunctible as well. Since compositions of ambiently $n$-dualizable (resp. $a^{(n-k)}$-adjunctible) morphisms are again ambiently $n$-dualizable (resp. $a^{(n-k)}$-adjunctible), the maps $\Phi_1$ and $\Phi_1^{-1}$
preserve these properties.
\end{proof}

Let $a,b:X \rightarrow Y$ be ambiently $n$-dualizable 1-morphims in $\Cc$. Let $a' := \Phi_1(a)$ and $b' := \Phi_1(b)$. It follows from Lemma \ref{lem:gen-fold-1} that $a'$ and $b'$ are ambiently $n$-dualizable. Let $(a'\dashv (a')^R, u_{a'},v_{a'})$ be duality data of an adjunction. Define a map
\begin{align*}
    \Phi_2 : \Hom_{\Cc}(a',b') &\to \Hom_{\Cc}(\Id_{\unit},  (a')^R \circ b' ) \\
    f':a' \to b' &\mapsto   (\Id_{(a')^R} \circ_1 f' ) \circ_2 u_{a'},
\end{align*}
where $\circ_1$ and $\circ_2$ denote the composition of 1-morphisms and 2-morphisms, respectively. A key observation is that $(a')^R \circ b'$ is a 1-morphism from $\unit$ to $\unit$ in $\Cc$. Consequently, for a $k$-morphism $f'$ in $\Cc$ that belongs to $\Hom_{\Cc}(a',b')$, the $k$-morphism $\Phi_2(f')$ in $\Cc$ may be regarded as a $(k-1)$-morphism in the loop category $\Omega\Cc$.

\begin{lemma}[Folding Lemma II]\label{lem:gen-fold-2}
The functor $\Phi_2$ is an equivalence. Let $f'$ be a $k$-morphism in $\Cc$ that belongs to $\Hom_{\Cc}(a',b')$. Then the following are equivalent:
\begin{enumerate}[label = (\arabic*)]
    \item $f'$ is ambiently $n$-dualizable (resp. $a^{(n-k)}$-adjunctible) in $\Cc$.
    \item $\Phi_2(f')$ is ambiently $n$-dualizable (resp. $a^{(n-k)}$-adjunctible) in $\Cc$.
    \item $\Phi_2(f')$ is ambiently $(n-1)$-dualizable (resp. $a^{(n-k)}$-adjunctible) in $\Omega\Cc$.
\end{enumerate}
\end{lemma}

\begin{proof}
An inverse to $\Phi_2$ is given by
\begin{align*}
    \Phi_2^{-1}:\Hom_{\Cc}(\Id_{\unit},  (a')^R \circ b' )   &\to \Hom_{\Cc}(a',b'),\\
                   g' &\mapsto  (v_{a'}  \circ_1 \Id_{b'}) \circ_2 (\Id_{a'} \circ_1 g'),
\end{align*}
The fact that this is an inverse follows from the zig-zag relations for the adjunction $a'\dashv (a')^R$. Since $a'$ belongs to $\Cc^{nd}$, so do $\Id_{a'^R}$, $u_{a'}$, and $v_{a'}$. Hence $\Phi_2$ and $\Phi_2^{-1}$ preserve ambient $n$-dualizability (resp. $a^{(n-k)}$-adjunctibility). This establishes the equivalence $(1) \Leftrightarrow (2)$.

The equivalence $(2) \Leftrightarrow (3)$ follows from the definition of $\Omega \Cc$. A $(k-1)$-morphism in $\Omega \Cc$ is ambiently $n$-dualizable (resp. $a^{(n-k)}$-adjunctible) as a $k$-morphism in $\Cc$ if and only if it is ambiently $(n-1)$-dualizable (resp. $a^{(n-k)}$-adjunctible) as a $(k-1)$-morphism in $\Omega \Cc$.
\end{proof}

We now prove Theorem \ref{thm:main-dual-result}.

\begin{proof}[Proof of Theorem \ref{thm:main-dual-result}]
Let $f$ be a $k$-morphism in $\Cc$ such that the source $S(f)$ and target $T(f)$ are ambiently $n$-dualizable. Let $a^{n-k}$ be a dexterity function of length $n-k$. If $f$ is ambiently $n$-dualizable, then $f$ is $a^{(n-k)}$-adjunctible. We must establish the converse. We proceed by induction.

The case that $k=1$ has already been established for all $n$ by Lemma \ref{lem:higher-left-right-lem} and Lemma \ref{lem:duality-induct-up}. Assume that $k>1$ and that the statement holds for $k-1$. Let $f$ be an $a^{n-k}$-adjunctible $k$-morphism in $\Cc$ whose source and target are ambiently $n$-dualizable. We need to show that $f$ is ambiently $n$-dualizable.

We first consider the case that $f$
already belongs to the loop category $\Omega \Cc$. In this case, $f$ is an $a^{(n-k)}$-adjunctible $(k-1)$-morphism in $\Omega \Cc$. By the inductive hypothesis $f$ is ambiently $(n-1)$-dualizable in $\Omega \Cc$ and hence ambiently $n$-dualizable in $\Cc$. 

Suppose that $f \notin \Omega \Cc$. Let $X, Y$ be the source and target objects of $f$, respectively, and $a,b$ be the source and target 1-morphisms, respectively. Since the source and target of $f$ belong to $\Cc^{nd}$, their source and target data do as well. In particular, $X,Y$ and $a,b$ belong to $\Cc^{nd}$. The Folding Lemmas imply that $\Phi_2(\Phi_1(f))$ is $a^{n-k}$-adjunctible as a $(k-1)$-morphism in $\Omega \Cc$. By the inductive hypothesis $\Phi_2(\Phi_1(f))$ is ambiently $(n-1)$-dualizable as a $(k-1)$-morphism in $\Omega \Cc$. Applying the Folding Lemmas again, we have that $f$ is ambiently $n$-dualizable in $\Cc$.
\end{proof}

\section{The defect cobordism hypothesis}\label{sec:defect-cob-hyp}

In this section, we provide a reformulation of the cobordism hypothesis for defects \cite[Section 4.3]{L09} in terms of symmetric monoidal (op)lax natural transformations. The results of \cite{JFS15} show that every defect in the sense of \cite{L09} gives rise to a system of (op)lax natural transformations. For the prototypical defects considered here, we make this construction precise and formulate Conjecture~\ref{conj:main-conj}, which asserts that this assignment is an equivalence --- that is, that the system of (op)lax natural transformations determines the defect. For the general singularities defined in \cite[Section 4.3]{L09}, we formulate the analogous Conjecture~\ref{conj:general-singular-reformulation}. Our main technical result, Theorem \ref{thm:main-dual-result}, is the key ingredient in showing that, assuming the ordinary cobordism hypothesis, this conjecture is equivalent to the cobordism hypothesis for defects.

This section is organized as follows. In Section~\ref{sec:def-structure} we introduce the prototypical defect structures we primarily consider, including tangential structures and labeling systems. In Section~\ref{sec:reformulation} we state the reformulation of the cobordism hypothesis for prototypical defects in terms of systems of (op)lax natural transformations and give a precise account of the $O(n-k)$-action on the space of ambiently $n$-dualizable $k$-morphisms. In Section~\ref{sec:comparison} we compare our reformulation with the specialization of the cobordism hypothesis for defects in \cite{L09} to prototypical defects, and show that, assuming the ordinary cobordism hypothesis, they are equivalent statements. In Section~\ref{sec:conical}, we extend this reformulation to the general singularities introduced in \cite{L09}.

\subsection{Prototypical defect structures}\label{sec:def-structure}

We begin by describing the class of defect structures considered throughout most of this section. We focus on these prototypical structures because their globular form makes the relationship between defects and systems of (op)lax natural transformations especially transparent. We give an explicit account of their geometry so that the subsequent discussion does not require familiarity with the general theory of singularity data.

We refer to a manifold equipped with such a structure as a defect manifold. The underlying stratification of the codimension-$k$ defect manifolds considered here is a flag of submanifolds
\begin{equation*}
 M_k \subset M_{k-1} \subset \cdots \subset M_1 \subset M_0 :=M,   
\end{equation*}
in which each inclusion $M_i\subset M_{i-1}$ has codimension one and is equipped with a framing of its normal bundle. Thus, $M_i$ has codimension $i$ in $M$. Locally, this models each stratum as an iterated domain wall. In particular, each codimension-$i$ stratum has a source and target codimension-$(i-1)$ stratum, reflecting the globular structure of an $i$-morphism.

These defect structures are special cases of the singularity data considered in \cite[Section~4.3]{L09}. They may also be viewed as a special class of structured conically smooth stratified spaces in the sense of \cite{AFT17}; compare in particular their examples of labeled hypersurfaces, framed submanifolds, and specified intersections. Closely related, particularly from the perspective of encoding higher-categorical data, are the vari-framed disk-stratified manifolds studied in \cite{AFR18}.

We first introduce domain wall structures (the codimension one case) and then generalize them to flag-stratified manifolds (the case of arbitrary codimension). We equip the strata with compatible tangential structures and labels, and combine these data into the notion of a defect manifold. Finally, we introduce a family of standard cubes associated to a prototypical defect; evaluation on these cubes will be central to the statements of the cobordism hypothesis for defects and its reformulation.

\subsubsection{Domain wall manifolds}\label{sec:dom-wall-mflds}

We begin by considering the prototypical example of a codimension one defect, the case of a domain wall.

\begin{definition}\label{def:dom-wall-manifold}
Let $M$ be a manifold. A \textit{domain wall structure} for $M$ consists of a codimension one submanifold $M_1 \subset M$ along with an isomorphism
\begin{equation}
   \phi_1: TM \bigg|_{M_1} \xrightarrow{\cong} \underline{\RR} \oplus TM_1 \label{eq:dom-wall-iso}
\end{equation}
of vector bundles over $M_1$ (let $\hat{n}_1$ denote the section of $TM|_{M_1}$ corresponding to the constant section determined by $1\in \RR$), and a continuous function
\begin{align}\label{eq:dom-wall-coloring}
    c : M \backslash M_1 \to \{A,B\}
\end{align}
such that $\hat{n}_1$ is outward normal to the region $c^{-1}(A)$ and inward normal to the region $c^{-1}(B)$. A local picture of this condition is provided in Figure \ref{fig:local-domain-wall}. We call $M_1$ the \textit{codimension 1-stratum} of $M$. A \textit{diffeomorphism} of manifolds with domain wall structure is a diffeomorphism that preserves the codimension one stratum and the data of \eqref{eq:dom-wall-iso} and \eqref{eq:dom-wall-coloring}.
\end{definition}

\begin{figure}[ht]
    \centering
    \includegraphics[width=0.2\textwidth]{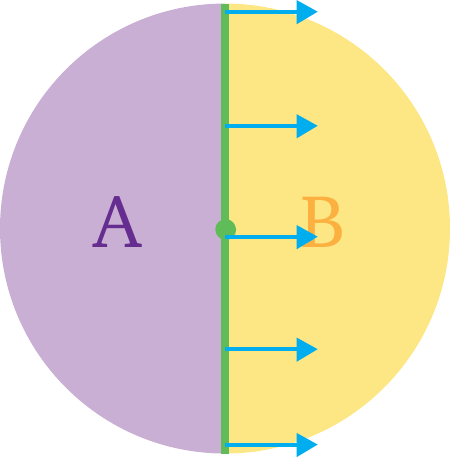}
    \caption{The local picture of domain wall structure at a point on the codimension one stratum. The codimension one stratum is colored green and the section $\hat{n}_1$ is depicted by blue tangent vectors.}
    \label{fig:local-domain-wall}
\end{figure}

\begin{remark}
The isomorphism of \eqref{eq:dom-wall-iso} determines a coorientation of the submanifold $M_1 \subset M$. Similarly, every collar neighborhood of $M_1$ in $M$ determines an isomorphism of the form \eqref{eq:dom-wall-iso}. Because we are working in a topological setting, these choices are equivalent up to contractible ambiguity. We choose to work with isomorphisms of the form \eqref{eq:dom-wall-iso} since this will be convenient for comparing tangential structures on the bulk manifold with tangential structures on the codimension 1-stratum.
\end{remark}

\subsubsection{Flag-stratified manifolds}\label{sec:flag-strat-manifolds}

We will generalize the notion of a manifold with domain wall structure to allow for manifolds carrying defect structures of arbitrary codimension. First, we introduce the stratified manifolds which we later equip with labels and tangential structures. We call such manifolds \emph{flag-stratified} since their stratification arises from a flag of submanifolds.

\begin{definition}
A \emph{flag-stratification of length $k$} on an $n$-dimensional manifold $M$ consists of submanifolds $M_i \subset M$ of codimension $i$ for $1\leq i \leq k$ such that $M_i \subset M_{i-1}$ (where $M_0 = M$), along with the data of isomorphisms
\begin{align}
    \phi_i : TM_{i-1}\bigg|_{M_{i}} \xrightarrow{\cong} \underline{\RR} \oplus TM_{i}, \label{eq:flag-iso}
\end{align}
for each $i\in\{1,..,k\}$. We adopt the convention that $M_{k+1}:=\emptyset$ and call $M_i\setminus M_{i+1}$ the codimension-$i$ stratum. We call $M_{i}$ the codimension-$\geq i$ stratum. A $n$-dimensional manifold equipped with a flag stratification of length $n$ is called an $n$-dimensional \emph{flag-stratified} manifold.
\end{definition}

\begin{remark}
Note that every flag-stratification of length $k < n$ extends to a flag stratification of length $n$ by setting $M_i = \emptyset$ for $i > k$. In this way every manifold equipped with a flag-stratification of length $k$ is a flag-stratified manifold.
\end{remark}

We simultaneously denote by $\hat{n}_i$ the section of the vector bundles
\begin{equation}
    TM_{i-1}\bigg|_{M_i} \subset TM_{i-2}\bigg|_{M_i}  \subset \cdots \subset TM\bigg|_{M_i}
\end{equation}
pulled back from the constant section $1\in \underline{\RR}$ by $\phi_i$. By composing the isomorphisms $\phi_{i}$ through $\phi_{1}$ we get an isomorphism
\begin{align}
   \varphi_i :  TM\bigg|_{M_i} \xrightarrow{\cong} \underline{\RR}_1 \oplus \cdots \oplus \underline{\RR}_i \oplus TM_i 
\end{align}
where $\underline{\RR}_j$ denotes the copy of $\underline{\RR}$ spanned by the image of the section $\hat{n}_j$.

\subsubsection{Defect tangential structures}\label{sec:bd-tang-str}

We now introduce defect tangential structures with which flag-stratified manifolds may be equipped. Recall from \cite{L09}, that an \textit{$n$-dimensional tangential structure} consists of a rank $n$-vector bundle $\zeta \to X$, and is denoted by the pair $(X,\zeta)$.\footnote{An $n$-dimensional tangential structure can also be thought of as a topological space equipped with a continuous map to $BO(n)$. Our definitions in this section also translate to this setting.}

\begin{definition}\label{def:defect-tang-str-n}
An \textit{$n$-dimensional defect tangential structure $\SingDat$ of length $k$} consists of the following data:
\begin{enumerate}
    \item For each $i \in\{0,...,k\}$ an $(n-i)$-dimensional tangential structure $(X_i,\zeta_i)$ such that $X_0$ is non-empty.
    \item For each $i \in \{1,...,k\}$ a pullback diagram
    \begin{equation} \label{eq:tang-str-pb-n}
    \begin{tikzcd}
    \underline{\RR} \oplus \zeta_i \arrow[r,"\tilde{q}"] \arrow[d]  & \zeta_{i-1} \arrow[d] \\
    X_i \arrow[r,"q"] & X_{i-1} 
    \end{tikzcd}
    \end{equation}
\end{enumerate}
We refer to $(X_0,\zeta_0)$ as the \textit{bulk tangential structure} and $(X_i,\zeta_i)$ as the \textit{codimension-$i$ tangential structure} for $i\geq 1$.
\end{definition}

\begin{remark}
As for flag-stratifications, every $n$-dimensional tangential structure of length $k < n$ extends to an $n$-dimensional tangential structure of length $n$, by setting $X_i = \emptyset$ for $i > k$. We say \emph{$n$-dimensional tangential structure} for an $n$-dimensional tangential structure of length $n$.
\end{remark}

This terminology is suggestive: we consider flag-stratified manifolds such that the codimension-$\geq i$ stratum is equipped with an $(X_i,\zeta_i)$-structure. These tangential structures fit together in a way described by the pull-back diagram \eqref{eq:tang-str-pb-n}.

Recall that if $(X,\zeta)$ is an $n$-dimensional tangential structure and $M$ is an $m$-dimensional manifold with $m\leq n$, then an $(X,\zeta)$-structure on $M$ consists of a pullback diagram
\begin{equation*}
\begin{tikzcd}
\widetilde{TM}^{n} \arrow[r,"\tilde{q}"] \arrow[d]
& \zeta \arrow[d] \\
M \arrow[r,"q"]
& X .
\end{tikzcd}
\end{equation*}
Here $\widetilde{TM}^{n} := \underline{\RR}^{\,n-m} \oplus TM$ denotes the inflated tangent bundle of $M$. Here, we inflate on the left so that the canonical cubes we define later are immediately equipped with the appropriate tangential structure. This corresponds to the idea that a codimension-$k$ defect defines a $k$-morphism, and placing the corresponding normal framing in the $k$-th position makes this connection more apparent.

For $M$ an $m$-dimensional flag stratified manifold, the tangential structures on the various strata are related as follows. Given an $(X_{i-1},\zeta_{i-1})$ structure on $M_{i-1}$, the splitting
\[
TM_{i-1}\big|_{M_i} \cong \underline{\RR} \oplus TM_i
\]
induces an $(X_{i-1},\zeta_{i-1})$-structure on $M_{i}$. We require this tangential structure to be compatible with the $(X_{i},\zeta_{i})$-structure on $M_{i}$ via the pullback diagram of \eqref{eq:tang-str-pb-n}.

\begin{definition}\label{def:M-defect-tang-str-n}
Let $\SingDat$ be an $n$-dimensional defect tangential structure. Let $M$ be an $m$-dimensional flag-stratified manifold. The data of an $\SingDat$-structure on $M$ consists of an $(X_i,\zeta_i)$-structure on $M_i$ for each $i\in\{0,...,m\}$ and the data of an isomorphism of $(X_{i-1},\zeta_{i-1})$-structures on $M_i$ making the following diagram commute
\begin{equation}\label{eq:tang-str-data-n}
    \begin{tikzcd}[row sep=small,column sep=small]
	{\underline{\RR} \oplus \widetilde{TM}_i^{n-i} } && {\underline{\RR} \oplus \zeta_{i}} \\
	& {\widetilde{TM}_{i-1}^{n-i+1}} && {\zeta_{i-1}} \\
	{M_i} && {X_i} \\
	& {M_{i-1}} && {X_{i-1}}
	\arrow[hook, from=3-1, to=4-2]
	\arrow[from=3-1, to=3-3]
	\arrow[from=4-2, to=4-4]
	\arrow[from=3-3, to=4-4]
	\arrow[from=1-1, to=3-1]
	\arrow[from=1-1, to=1-3]
	\arrow[from=2-2, to=2-4]
	\arrow[from=2-4, to=4-4]
	\arrow[from=1-3, to=3-3]
	\arrow[from=1-1, to=2-2]
	\arrow[from=2-2, to=4-2]
	\arrow[from=1-3, to=2-4]
\end{tikzcd}
\end{equation}
for $1 \leq i\leq m$.
\end{definition}

\begin{remark}\label{rem:tang-str-cond-n}
The data of \eqref{eq:tang-str-data-n} identifies the tangential structure on the codimension-$i$ stratum with the tangential structure on the codimension-$(i-1)$ stratum by the data of the pullback \eqref{eq:tang-str-pb-n}. This data rarely plays a pivotal role in our discussion and is often trivial. One could also state Definition \ref{def:M-defect-tang-str-n} with this data replaced by an equality condition by noting that any identification can be absorbed into a collar neighborhood of the codimension-$i$ stratum.
\end{remark}

\begin{remark}
The notion of a defect tangential structure is closely related to the treatment of tangential structures on conically smooth stratified spaces in \cite{AFT17}.   A related formulation of defect tangential structures appears in \cite{FMT22}, where compatibility between bulk and defect tangential structures is expressed over the sphere bundle of the normal bundle. The structures considered here are adapted to normally framed flags and record this compatibility successively along the inclusions $M_i\subset M_{i-1}$. For a more general formulation of tangential structures on stratified manifolds, see also \cite{Mul25}.
\end{remark}

\begin{remark}
The length of the defect tangential structure provides an upper limit for the depth of the flag. If $\SingDat$ is an $n$-dimensional defect tangential structure of length $k$, then an $n$-dimensional flag-stratified manifold $M$ with $\SingDat$-structure must have $M_i = \emptyset$ for all $i>k$.
\end{remark}

\begin{example}[Framed defects]\label{ex:framings-n}
    Consider the $n$-dimensional defect tangential structure given by $X_i = \ast$ for each $i$ and $\zeta_i = \RR^{n-i}$, with the maps \eqref{eq:tang-str-pb-n} given by the obvious isomorphism $\RR \oplus \RR^{n-i}  \cong \RR^{n-i+1}$. The data of an $\SingDat$-structure on an $n$-dimensional flag-stratified manifold $M$ consists of an $n$-framing on $M$ such that, when restricted to $M_i$, the $i$-th component of the framing agrees with $\hat{n}_i$.
\end{example}

\begin{example} Recall that given a topological group $G$ equipped with a homomorphism $\rho: G \to O(n)$, one can construct an $n$-dimensional tangential structure by taking the associated bundle $EG \times_{\rho} \RR^{n}$ over $BG$. Extending this example, given topological groups $G_i$ for  $0 \leq i \leq k$, equipped with homomorphisms $\rho_i : G_i \to O(n-i)$, along with commutative diagrams (not necessairly pullback diagrams)
\begin{equation}\label{eq:G-str-comm-diag}
    \begin{tikzcd}
    G_{i} \arrow[r,"\rho_{i}"] \arrow[d,"\varphi_i"]  & O(n-i) \arrow[d,hookrightarrow] \\
    G_{i-1}\arrow[r,"\rho_{i-1}"] & O(n-i+1)
\end{tikzcd}
\end{equation}
one can produce an $n$-dimensional defect tangential structure of length $k$. The codimension-$i$ tangential structure is given by the associated bundle $\zeta_i := EG_i \times{\rho_i} \RR^{n-i}$ over $BG_i$. The map $\varphi_i:G_i\to G_{i-1}$ induces a map $q_i:BG_i \to BG_{i-1}$. The commutative diagram \eqref{eq:G-str-comm-diag} induces an isomorphism of associated vector bundles
\begin{align}
    \zeta_i \times_{\rho_i} \RR^{n-i+1} \cong \zeta_{i} \times_{\rho_{i-1} \circ \varphi_i} \RR^{n-i+1}
\end{align}
over $BG_{i}$. The right hand side is equivalent to $q_i^*\zeta_{i-1}$ while the left hand side is equivalent to $\underline{\RR} \oplus \zeta_i$. This produces the pullback diagram
\begin{equation} 
    \begin{tikzcd}
    \underline{\RR} \oplus \zeta_i \arrow[r,"\tilde{q}"] \arrow[d]  & \zeta_{i-1} \arrow[d] \\
    BG_i \arrow[r,"q_i"] & BG_{i-1} 
    \end{tikzcd}
\end{equation}

We denote the corresponding defect tangential structure by $\overrightarrow{G}$.  A $\overrightarrow{G}$-structure on a flag-stratified manifold consists of a $G_i$-structure on the codimension-$i$ strata, such that these structures are related by the diagrams of \eqref{eq:G-str-comm-diag}.

In the case that $G$ is a topological group equipped with a homomorphism $\phi: G \to O(n)$. We can produce an $n$-dimensional defect tangential structure by iteratively defining $G_i : = G_{i-1} \times_{O(n-i+1)} O(n-i)$, where $G_0 = G$. Moreover, given a stable tangential structure one can produce an $n$-dimensional defect tangential structure for each $n$. For example, orientations and spin structures.
\end{example}

\subsubsection{Defect labels}\label{sec:def-labels}

We now introduce the notion of a labeling system for a flag-stratified manifold. A labeling system is a globular space whose $i$-cells are defect labels of codimension $i$. The globularity conditions ensure that the source and target labels of a codimension-$i$ defect are compatible with the source and target labels of the adjacent codimension-$(i-1)$ defects.

\begin{definition}\label{def:labeling-data}
A \textit{labeling system $\overrightarrow{L}$ of length $k$} consists of topological spaces $L_{i}$ for $0\leq i \leq k$, and for each $i>0$ a pair of continuous maps $s_i,t_i:L_i \to L_{i-1}$, satisfying
\begin{align*}
    s_i \circ t_{i+1} = s_i \circ s_{i+1} \qquad \text{and} \qquad t_i \circ  t_{i+1} = t_i \circ s_{i+1} \, ,
\end{align*}
for $1\leq i< k$.
\end{definition}

We call $L_0$ the space of \textit{bulk-theory labels} and for each $i\geq 1$ we call $L_i$ the space of codimension-$i$ \textit{defect labels}. The functions $s_i,t_i:L_i \to L_{i-1}$ will determine the codimension-$(i-1)$ labels on either side of the codimension-$i$ strata. It is useful to keep in mind the case that each $L_i$ is discrete (equivalently, $\overrightarrow{L}$ is a globular set). If we want to consider this case in isolation, we will call it a \textit{discrete labeling}.

\begin{definition}\label{def:mfld-with-defects}
Let $\overrightarrow{L}$ be a labeling of length $k$ and $M$ a manifold equipped with a flag-stratification of length $k$. An $\overrightarrow{L}$-labeling on a manifold $M$ consists of maps
\begin{align*}
    &c_i : M_{i}\backslash M_{i+1} \to L_i 
\end{align*}
for $0\leq i\leq k$, such that the following condition holds:
\begin{itemize}
    \item[1)] For $1\leq i \leq k$ and $x\in M_i$, the vector $\hat{n}_i$ points from $s_i(c_i(x))$ to $t_i(c_i(x))$ in the following sense: let $\widetilde{n}_i$ be an extension of $\hat{n}_i$ to an open neighborhood of $M_i$ in $M_{i-1}$ and let
\[
\phi_{\tilde{n}_i}^{x}:(-\delta,\epsilon)\to M_{i-1}
\]
be the flow of $\tilde{n}_i$ through $x$. We require that
\begin{align*}
    \lim_{t\to 0^-}
    c_{i-1}\!\left(\phi_{\tilde{n}_i}^{x}(t)\right)
    = s_i(c_i(x)), \qquad \text{and} \qquad 
    \lim_{t\to 0^+}
    c_{i-1}\!\left(\phi_{\tilde{n}_i}^{x}(t)\right)
    = t_i(c_i(x)).
\end{align*}
\end{itemize}

\end{definition}

We refer to the maps $c_i$ as \textit{labeling maps}. In the case that $\overrightarrow{L}$ is discrete, the map $c_i$ assigns a label to each connected component of the codimension-$i$ strata and the label either side of the codimension-$i$ strata corresponds to the source and target label.

\begin{example}\label{ex:dom-wall-labels}
Let $\overrightarrow{L}$ be the labeling of length one given by $L_0 = \{A,B\}$ and $L_1 = \{\ast\}$, equipped with the map $s_1(\ast) = A$ and $t_1(\ast) = B$. Then, a manifold with $\overrightarrow{L}$-labels is a manifold with domain wall structure. %We denote this labeling by $\dom$.
\end{example}

\begin{example}\label{ex:codim-k-defect-labeling}
More generally, let $\overrightarrow{L}$ be the labeling system of length $k$ such that
\begin{align*}
L_i &= \{A_i,B_i\} \qquad \text{for } i<k, \qquad \text{and} \qquad L_k = \{\ast\}.
\end{align*}
The source and target maps are given by
\begin{align*}
s_i(A_i) &= A_{i-1}, &
s_i(B_i) &= A_{i-1}, \\
t_i(A_i) &= B_{i-1}, &
t_i(B_i) &= B_{i-1},
\end{align*}
for $1\leq i<k$, together with $s_k(\ast) = A_{k-1}$ and $t_k(\ast) = B_{k-1}$.

This is the labeling system associated with a prototypical codimension-$k$ defect. The unique element $\ast\in L_k$ represents a codimension-$k$ defect whose source and target are the codimension-$(k-1)$ defects $A_{k-1}$ and $B_{k-1}$, respectively.
\end{example}

\begin{remark}
In the discrete case, our labeling systems may be compared with the defect data of \cite{CRS19}, which assign labels to strata together with incidence data encoded by decorated links. For the normally framed flags considered here, this incidence information is encoded by the globular source and target maps.
\end{remark}

\subsubsection{Defect manifolds}

We now bring together the previous sections to provide a definition of a defect manifold.

\begin{definition}\label{def:defect-datum}
An \textit{$n$-dimensional defect datum $\defect$ of length $k$} consists of a pair $(\SingDat,\Lab)$, where $\SingDat$ is an $n$-dimensional defect tangential structure of length $k$ and $\Lab$ is a labeling of length $k$. An \textit{$n$-dimensional defect datum $\defect$} is an $n$-dimensional defect datum of length $n$.
\end{definition}

\begin{definition}\label{defn:manifold-with-defects}
Let $\defect = (\SingDat,\Lab)$ be an $n$-dimensional defect datum of length $k$ and let $M$ be a flag-stratified manifold of dimension $m \leq n$. A \textit{$\defect$-structure on $M$} consists of an $\SingDat$-structure on $M$ and an $\Lab$-labeling on $M$. A \emph{$\defect$-manifold} is a flag-stratified manifold equipped with a $\defect$-structure. We refer to $\defect$-manifolds collectively as defect manifolds.
\end{definition}

\begin{remark}
In our setup, the labeling data and tangential structure data are independent. In particular, no compatibility conditions are imposed between the labeling system $\Lab$ and the defect tangential structure $\SingDat$.
\end{remark}

\subsubsection{Standard cubes}\label{sec:standard-cubes}

We define the standard cubes, which will be central to the statement of the cobordism hypothesis for defects. Let $\Def=(\SingDat,\Lab)$ be an $n$-dimensional defect datum of length $k$. The standard cube $\mathbb{I}^k$ gives rise to a family of distinguished $k$-morphisms in the bordism $(\infty,n)$-category $\Bord_n(\defect)$ of $\defect$-manifolds.

For the purposes of this section, it suffices to work with the usual description of bordisms in terms of manifolds with corners. In particular, we regard the cube $[-1,1]^k$ as a bordism from $[-1,1]^{k-1}\times\{-1\}$ to $[-1,1]^{k-1}\times\{1\}$ and equip it with a canonical $\Def$-structure. For more details on the relationship between manifolds with corners and bordisms, we refer the reader to \cite{SP14,CS19,FT21}.

\begin{definition}
Let $\kcube$ be the $k$-dimensional flag-stratified manifold $[-1,1]^k$ with flag-stratification given by
\begin{align*}
     \{0\}^{k} \subset  \{0\}^{k-1} \times [-1,1] \subset \cdots \subset \{0\} \times [-1,1]^{k-1} \subset [-1,1]^{k},
\end{align*}
and whose isomorphisms \eqref{eq:flag-iso} are determined by $\hat{n}_i = \partial_{i}$.
\end{definition}

We first construct a family of $\SingDat$-structures on $\kcube$. Let $\widetilde{X}_k$ be the frame bundle of $\zeta_k \to X_k$. An element $\tilde{x} \in \widetilde{X}_k$ is equivalent to an $(X_k,\zeta_k)$-structure on the point, in other words, a pullback diagram
\begin{equation}\label{eq:X-structure-on-pt}
    \begin{tikzcd}[row sep=small,column sep=small]
	{\RR^{n-k}} && {\zeta_k} \\
	\\
	\pt && {X_{k}}
	\arrow[from=1-1, to=1-3]
	\arrow[from=1-1, to=3-1]
	\arrow[from=1-3, to=3-3]
	\arrow[from=3-1, to=3-3]
\end{tikzcd}
\end{equation}
For each $\widetilde{x}\in \widetilde{X}_k$ we equip $\kcube$ with a $\SingDat$-structure as follows. The $k$-stratum of $\kcube$ is a point, which we equip with the $(X_k,\zeta_k)$-structure determined by $\widetilde{x}\in \widetilde{X}_k$. The diagrams of \eqref{eq:tang-str-pb-n} induce maps $\tilde{q}_i : \widetilde{X}_i \to \widetilde{X}_{i-1}$. Equip the $i$-stratum $\{0\}^{i} \times [-1,1]^{k-i}$ with the $(X_i,\zeta_i)$-structure determined by the $(\tilde{q}_{i+1}\circ \cdots \circ \tilde{q}_k)(\tilde{x}) \in \widetilde{X}_i$. More concretely, the $(X_i,\zeta_i)$-structure on the $i$-stratum is given by the composition of pullbacks
\begin{equation*}
    \begin{tikzcd}[row sep=small,column sep=small]
	{T([-1,1]^{k-i}) \oplus \underline{\RR}^{n-k} } && {\RR^{k-i} \oplus \RR^{n-k}} && { \underline{\RR}^{k-i} \oplus \zeta_k } && {\zeta_{i}} \\
	\\
	{[-1,1]^{k-i}} && \pt && {X_k} && {X_{i}}
	\arrow[from=1-1, to=1-3]
	\arrow[from=1-1, to=3-1]
	\arrow[from=1-3, to=1-5]
	\arrow[from=1-3, to=3-3]
	\arrow[from=1-5, to=1-7]
	\arrow[from=1-5, to=3-5]
	\arrow[from=1-7, to=3-7]
	\arrow[two heads, from=3-1, to=3-3]
	\arrow[from=3-3, to=3-5]
	\arrow[from=3-5, to=3-7]
\end{tikzcd}
\end{equation*}
where the pullback on the right is obtained using diagrams \eqref{eq:tang-str-pb-n} in the definition of $\SingDat$. The diagrams \eqref{eq:tang-str-data-n} in the definition of an $\SingDat$-structure commute by construction. Hence, for each $\widetilde{x} \in \widetilde{X}_k$, the flag-stratified manifold $\kcube$ admits a canonical $\SingDat$-structure.

Let $\lambda \in L_k$, define an $\Lab$-labeling on $\kcube$ by setting $c_k(\{0\}^k)= \lambda$ and
\begin{align*}
    c_i(x) = \begin{cases}
        (s_{i+1}\circ \cdots \circ s_{k})(\lambda) & x_{i+1} < 0\\
        (t_{i+1} \circ \cdots \circ t_k)(\lambda) & x_{i+1} > 0
    \end{cases}
\end{align*}   
for $i<k$.

\begin{definition}
Let $\Def=(\SingDat,\Lab)$ be a defect datum of length $k$. The \emph{standard $k$-cube} $\kcube(\tilde x,\lambda)$ associated to $(\tilde x,\lambda) \in \widetilde{X}_k \times L_k$ is the flag stratified manifold $\kcube$ equipped with the $\SingDat$-structure determined by $\tilde{x}\in \widetilde X_k$ and the labeling determined by $\lambda \in L_k$.
\end{definition}

For $0\leq i\leq k$, applying this construction to the truncation of $\Def$ to length $i$ defines a standard $i$-cube $\Icube{i}(\widetilde{x},\lambda)$ for $(\widetilde{x},\lambda)\in\widetilde X_i\times L_i$. Moreover, the source and target of $\kcube(\tilde x,\lambda)$ are $\Icube{k-1}(\widetilde{q}_k(\widetilde x),s_k(\lambda))$ and $\Icube{k-1}(\widetilde{q}_k(\widetilde x),t_k(\lambda))$, respectively.

\subsection{Reformulation via oplax natural transformations}\label{sec:reformulation}

We now formulate a description of field theories with prototypical defects in terms of systems of symmetric monoidal (op)lax natural transformations. Such a system is encoded by symmetric monoidal functors valued in the (op)lax arrow categories of the target. More precisely, each codimension-$i$ defect label determines a field theory valued in the (op)lax arrow category of $i$-morphisms, while the compatibility between defects of successive codimensions is expressed through the corresponding source and target functors. Throughout, we work primarily with oplax arrow categories and oplax natural transformations. Similar statements and results hold in the lax setting, which we will indicate in subsequent remarks. 

We begin by recalling the oplax arrow categories introduced in \cite{JFS15}. We then give a precise account of the $O(n-k)$-action on the space of ambiently $n$-dualizable $k$-morphisms. Finally, we formulate Conjecture~\ref{conj:main-conj}, which asserts that the resulting system of symmetric monoidal oplax natural transformations determines the original defect theory.

\subsubsection{Oplax arrow categories}

Let $\Cc$ be a symmetric monoidal $(\infty,n)$-category. For $1\leq k\leq n$, let $\Cc^{\to}_{(k)}$ denote the \emph{oplax arrow category} of $k$-morphisms  as defined in \cite{JFS15}. The objects of $\Cc^{\to}_{(k)}$ are the $k$-morphisms of $\Cc$. We will not need to utilize the technical definition of $\Cc^{\to}_{(k)}$; however, it is useful to keep in mind the following informal description of $\Cc^{\to}_{(1)}$:
\begin{itemize}
    \item The objects of $\Cc_{(1)}^{\to}$ are the 1-morphisms of $\Cc$ denoted by a horizontal arrow
    \begin{equation*}
        a \xrightarrow{f} a'
    \end{equation*}
    \item Given a pair of objects $f:a\to a'$ and $g:b\to b'$ in $\Cc_{(1)}^{\to}$, a 1-morphism from $f$ to $g$ in $\Cc_{(1)}^{\to}$ consists of a pair of 1-morphisms $h:a\to b$ and $h':a' \to b'$ and a 2-morphism $u:g\circ h \to h'\circ f$ in $\Cc$. This is depicted by the following diagram
    \begin{equation*}
          \begin{tikzcd}[row sep=small,column sep=small]
	   a && {a'} \\
	   \\
	   b && {b'}
	   \arrow["f", from=1-1, to=1-3]
	   \arrow["h"',from=1-1, to=3-1]
	   \arrow["h'",from=1-3, to=3-3]
	   \arrow["u",between={0.2}{0.8}, Rightarrow, from=3-1, to=1-3]
	   \arrow["g"', from=3-1, to=3-3]
    \end{tikzcd}
    \end{equation*}
    Composition corresponds to composing diagrams vertically.

    \item For $i\leq n$, an $i$-morphism in $\Cc_{(1)}^{\to}$ consists of a pair of $i$-morphisms in $\Cc$ together with a not necessarily invertible $(i+1)$-morphism filling out a higher coherence cell.
    
    For example, given a pair of 1-morphisms $(h,h',u)$ and $(j,j',v)$ from $f$ to $g$ in $\Cc_{(1)}^{\to}$, a 2-morphism from $(h,h',u)$ to $(j,j',v)$ consists of a pair of 2-morphisms $\nu : h \to j$ and $\nu':h'\to j'$ in $\Cc$, along with a 3-morphism $\gamma$ in $\Cc$ of the form
    \begin{equation*}
        \begin{tikzcd}
	a && {a'} &&& a && {a'} \\
	\\
	b && {b'} &&& b && {b'}
	\arrow["f", from=1-1, to=1-3]
	\arrow[""{name=0, anchor=center, inner sep=0}, "j", curve={height=-18pt}, from=1-1, to=3-1]
	\arrow[""{name=1, anchor=center, inner sep=0}, "h"', curve={height=18pt}, from=1-1, to=3-1]
	\arrow[""{name=2, anchor=center, inner sep=0}, "{j'}", curve={height=-18pt}, from=1-3, to=3-3]
	\arrow["f", from=1-6, to=1-8]
	\arrow[""{name=3, anchor=center, inner sep=0}, "h"', curve={height=18pt}, from=1-6, to=3-6]
	\arrow[""{name=4, anchor=center, inner sep=0}, "{j'}", curve={height=-18pt}, from=1-8, to=3-8]
	\arrow[""{name=5, anchor=center, inner sep=0}, "{h'}"', curve={height=18pt}, from=1-8, to=3-8]
	\arrow["{v}"', curve={height=12pt}, between={0.3}{0.8}, Rightarrow, from=3-1, to=1-3]
	\arrow["g"', from=3-1, to=3-3]
	\arrow["u", curve={height=-12pt}, between={0.2}{0.8}, Rightarrow, from=3-6, to=1-8]
	\arrow["g"', from=3-6, to=3-8]
	\arrow["\nu", between={0.2}{0.8}, Rightarrow, from=1, to=0]
	\arrow["\gamma", between={0.3}{0.7}, Rightarrow, scaling nfold=3, from=2, to=3]
	\arrow["{\nu'}", between={0.2}{0.8}, Rightarrow, from=5, to=4]
    \end{tikzcd}
    \end{equation*}
    
    \item The $i$-morphisms in $\Cc_{(1)}^{\to}$ for $i > n $  are equivalences in $\Cc$ between diagrams describing $(i-1)$-morphisms in $\Cc_{(1)}^{\to}$.
\end{itemize}

The oplax arrow categories are equipped with symmetric monoidal functors
\begin{align*}
    S_k,T_k:\Cc^{\to}_{(k)} \to \Cc_{(k-1)}^{\to},
\end{align*}
(where $\Cc^{\to}_{(0)} := \Cc$) called the \emph{source} and \emph{target} functors, satisfying $S_{k-1}\circ S_k = S_{k-1} \circ T_k$ and $T_{k-1}\circ S_k = T_{k-1} \circ T_k$.

In the case that $\Cc$ is symmetric monoidal, the oplax arrow categories $\Cc^{\to}_{(k)}$ inherit a symmetric monoidal structure, as do the source and target functors to $\Cc^{\to}_{(k-1)}$. 

Every $i$-morphism $g$ in $\Cc^{\to}_{(k)}$ has an underlying $(k+i)$-morphism in $\Cc$ (the one filling out the diagram), which is denoted $g^{\sharp}$. Note that the assignment $g \mapsto g^{\sharp}$ is not functorial. In particular, it does not preserve the source and target.

The following criterion is the key property of $\Cc^{\to}_{(k)}$ that we will use.

\begin{theorem}[{Johnson-Freyd--Scheimbauer \cite[Theorem~7.6]{JFS15}}]\label{thm:JFS-oplax-dual}
An object $f$ of $\Cc^{\to}_{(k)}$ is $(n-k)$-dualizable in $\Cc^{\to}_{(k)}$ if and only if $S_k(f)$ and $T_k(f)$ are $(n-k)$-dualizable in $\Cc^{\to}_{(k-1)}$, and $f^{\sharp}$ is $(n-k)$-times right-adjunctible.
\end{theorem}

\begin{remark}
The symmetric monoidal lax arrow category of $i$-morphisms in $\Cc$ is denoted by $\Cc^{\downarrow}_{(i)}$ and is similarly equipped with systems of source and target functors. An object $f$ of $\Cc^{\downarrow}_{(k)}$ is $(n-k)$-dualizable in $\Cc^{\downarrow}_{(k)}$ if and only if $S_k(f)$ and $T_k(f)$ are $(n-k)$-dualizable in $\Cc^{\downarrow}_{(k-1)}$, and $f^{\sharp}$ is $(n-k)$-times left-adjunctible.
\end{remark}

\subsubsection{The $O(n-k)$-action}

We record here an action of $O(n-k)$ on the space of ambiently $n$-dualizable $k$-morphisms in a symmetric monoidal $(\infty,n)$-category $\Cc$. Such an action is asserted by Lurie in the paragraph preceding \cite[Theorem 4.3.11]{L09}, and, as noted in \cite{Dagger25}, follows from the results of \cite[\S 7]{JFS15}. We provide here a proof for general $k$.

\begin{definition}\label{def:n-def-dual-cat}
Let $\Cc$ be a symmetric monoidal $(\infty,n)$-category. The
\emph{$(\infty,n-k)$-category of ambiently $n$-dualizable $k$-morphisms} is
\begin{equation*}
    \DefCat_k^n(\Cc) := \iota_{n-k} ((\Cc^{nd})^{\to}_{(k)}),
\end{equation*}
obtained by truncating $(\Cc^{nd})^{\to}_{(k)}$ to an
$(\infty,n-k)$-category.
\end{definition}

The underlying $\infty$-groupoid of $\DefCat_k^n(\Cc)$ is the space of ambiently $n$-dualizable $k$-morphisms in $\Cc$. The following theorem establishes the required duals.

\begin{theorem}\label{thm:O(n-k)-action}
Let $\Cc$ be a symmetric monoidal $(\infty,n)$-category with $n$-duals and let $0 \leq k \leq n$. Then $\Cc^{\to}_{(k)}$ has $(n-k)$-duals.
\end{theorem}

\begin{proof}
Throughout the proof, we use the dualizability and adjunctibility criteria for morphisms in the oplax arrow categories, as established in \cite[Section~7]{JFS15}. The case $k=0$ is true by assumption. We proceed by induction on $k$, so assume that $\Cc^{\to}_{(k-1)}$ has $(n-k+1)$-duals for some $1\leq k \leq n-1$. Note that the case $k=n$ is vacuous.

We first treat objects. An object $f$ is 1-dualizable in $\Cc^{\to}_{(k)}$ if and only if $S_k(f)$ and $T_k(f)$ are 1-dualizable in $\Cc^{\to}_{(k-1)}$, and $f^{\sharp}$ is right-adjunctible in $\Cc$. Both conditions hold: the former by the inductive hypothesis, and the latter since $f^{\sharp}$ is a $k$-morphism of $\Cc$ with $k < n$ and $\Cc$ has $n$-duals.

Next, we show that every $i$-morphism $f$ in $\Cc^{\to}_{(k)}$ for $1\leq i < n-k$ is right-adjunctible. An $i$-morphism $f$ is right-adjunctible in $\Cc^{\to}_{(k)}$ if and only if $S_k(f)$ and $T_k(f)$ are right-adjunctible in $\Cc^{\to}_{(k-1)}$ and the underlying bulk $(k+i)$-morphism $f^{\sharp}$ is right-adjunctible in $\Cc$. Again, both conditions hold: the former by the inductive hypothesis, and the latter since $f^\sharp$ is a $(k+i)$-morphism in $\Cc$ with $k +i <n$, and $\Cc$ has $n$-duals.

It remains to establish left-adjunctibility for an $i$-morphism $f$ in $\Cc^{\to}_{(k)}$ for $1\leq i < n-k$. An $i$-morphism $f$ is left-adjunctible in $\Cc^{\to}_{(k)}$ if and only if $S_k(f)$ and $T_k(f)$ are left-adjunctible in $\Cc^{\to}_{(k-1)}$ and a certain mate of the underlying bulk $(k+i)$-morphism $f^{\sharp}$ is left-adjunctible in $\Cc$. The former holds by the inductive hypothesis. The latter holds because this mate is again a $(k+i)$-morphism of $\Cc$, with $k+i<n$, and $\Cc$ has $n$-duals.

Thus, every object of $\Cc^\to_{(k)}$ is dualizable and every $i$-morphism for $1\leq i<n-k$ has both left and right adjoints. This completes the inductive step.
\end{proof}

\begin{remark}
Alternatively, after establishing right-adjunctibility, the existence of left adjoints can be deduced from \Cref{thm:main-dual-result}. Indeed, \Cref{thm:JFS-oplax-dual}, together with the inductive hypothesis, shows that every object of $\Cc^\to_{(k)}$ is $(n-k)$-dualizable, while the preceding argument implies recursively that every $i$-morphism is $(n-k-i)$-times right-adjunctible. Applying \Cref{thm:main-dual-result} inductively in $i$ then shows that every $i$-morphism for $1\leq i<n-k$ is ambiently $(n-k)$-dualizable, and hence left-adjunctible.
\end{remark}

\begin{corollary}\label{cor:O(n-k)-action}
Let $\Cc$ be a symmetric monoidal $(\infty,n)$-category and let $0\leq k\leq n$. The space of ambiently $n$-dualizable $k$-morphisms in $\Cc$ carries a canonical $O(n-k)$-action.
\end{corollary}

\begin{proof}
Applying \Cref{thm:O(n-k)-action} to $\Cc^{nd}$ shows that $(\Cc^{nd})^\to_{(k)}$, and hence its $(n-k)$-truncation $\DefCat_k^n(\Cc)$, has $(n-k)$-duals. The ordinary cobordism hypothesis therefore equips its underlying $\infty$-groupoid, which is the space of ambiently $n$-dualizable $k$-morphisms in $\Cc$, with an $O(n-k)$-action.
\end{proof}

\begin{remark}
Theorem~\ref{thm:O(n-k)-action} and Corollary~\ref{cor:O(n-k)-action} also apply to the lax arrow categories in place of the oplax arrow categories. The arguments in the proof of Theorem~\ref{thm:O(n-k)-action} exchange the roles of right and left adjoints.
\end{remark}

\subsubsection{Statement of the reformulation}

We now propose a reformulation of the cobordism hypothesis for the prototypical defect types introduced in Section \ref{sec:def-structure}. Let $\defect = (\SingDat,\Lab)$ be an $n$-dimensional defect datum of length $k$ as in Definition \ref{def:defect-datum}. For $1\leq i \leq k$, the pullback diagram
\begin{equation*}
    \begin{tikzcd}
    \underline{\RR} \oplus \zeta_i \arrow[r,"\tilde{q}"] \arrow[d]  & \zeta_{i-1} \arrow[d] \\
    X_i \arrow[r,"q"] & X_{i-1} 
    \end{tikzcd}
\end{equation*}
induces a symmetric monoidal functor
\begin{align*}
    \delta_i : \Bord_{n-i}^{(X_{i},\zeta_{i}) } \to \Bord_{n-i+1}^{(X_{i-1},\zeta_{i-1}) }
\end{align*}
Similarly, for each $1\leq i \leq k$, there are maps $s_i,t_i:L_{i} \to L_{i-1}$ taking a label in $L_i$ to its source and target labels. 

\begin{conjecture}\label{conj:main-conj}
Let $\Def$ be an $n$-dimensional defect datum of length $k$. Let $\Cc$ be a symmetric monoidal $(\infty,n)$-category. The following data are equivalent:
\begin{enumerate}
    \item Symmetric monoidal functors $Z:\Bord_{n}(\Def) \to \Cc$.
    \item For each $0\leq i \leq k$, a family of symmetric monoidal functors
    \begin{align*}
         \{\beta_{\lambda}:\Bord_{n-i}^{(X_i,\zeta_i)} \to \Cc^{\to}_{(i)}\}_{\lambda\in L_i}  
    \end{align*}
    such that the following diagrams commute:
    \begin{equation}\label{eq:oplax-source-target-diags}
       \begin{tikzcd}[row sep=small,column sep=small]
	{\Bord_{n-i}^{(X_{i},\zeta_{i})}} && {\Cc^{\to}_{(i)}} \\
	\\
	{\Bord_{n-i+1}^{(X_{i-1},\zeta_{i-1})}} && {\Cc^{\to}_{(i-1)}}
	\arrow["{\beta_{\lambda}}", from=1-1, to=1-3]
	\arrow["{\delta_i}"', from=1-1, to=3-1]
	\arrow["{S_i}", from=1-3, to=3-3]
	\arrow["{\beta_{s_i(\lambda)}}"', from=3-1, to=3-3]
    \end{tikzcd} 
        \qquad \begin{tikzcd}[row sep=small,column sep=small]
	{\Bord_{n-i}^{(X_{i},\zeta_{i})}} && {\Cc^{\to}_{(i)}} \\
	\\
	{\Bord_{n-i+1}^{(X_{i-1},\zeta_{i-1})}} && {\Cc^{\to}_{(i-1)}}
	\arrow["{\beta_{\lambda}}", from=1-1, to=1-3]
	\arrow["{\delta_i}"', from=1-1, to=3-1]
	\arrow["{T_i}", from=1-3, to=3-3]
	\arrow["{\beta_{t_i(\lambda)}}"', from=3-1, to=3-3]
    \end{tikzcd}
    \end{equation}
    for every $1\leq i \leq k$ and $\lambda \in L_i$.
\end{enumerate}
Moreover, for $0\leq i \leq k$, this equivalence satisfies $Z(\Icube{i}(\tilde{x},\lambda)) = \beta_{\lambda}(\pt_{\tilde{x}})$, where $\tilde{x} \in \widetilde{X}_i$ is an element in the frame bundle of $\zeta_i$, $\lambda \in L_i$, and $\Icube{i}(\tilde{x},\lambda)$ is the standard cube introduced in Section \ref{sec:standard-cubes} .
\end{conjecture}

\begin{remark}
The equivalence of Conjecture \ref{conj:main-conj} is $O(n-i)$-equivariant for all $0\leq i\leq k$ in the following way. The group $O(n-i)$ acts on $\Icube{i}(\tilde{x},\lambda)$ via its action on $\tilde{x} \in \widetilde{X}_i$. The $O(n-i)$-action on the family $\{\beta_{\lambda}\}_{\lambda \in L_i}$ is induced by its action on $(X_i,\zeta_i)$-structures. The relation $Z(\Icube{i}(\tilde{x},\lambda)) = \beta_{\lambda}(\pt_{\tilde{x}})$ ensures that these actions coincide.
\end{remark}

\begin{remark}\label{rem:infty-groupoid-formulation}
Consider the case that $\Def^{\fr}_1$ is the defect datum corresponding to a framed domain wall. Then \Cref{conj:main-conj} says that a functor $\Bord_{n}(\Def^{\fr}_1) \to \Cc$ is equivalent to a pair of topological symmetric monoidal functors $F_{A},F_B :\Bord_{n}^{\fr} \to \Cc$, and a symmetric monoidal functor $\beta:\Bord_{n-1}^{\fr}\to \Cc^{\to}_{(1)}$ such that $S \circ \beta  = \tau_{\leq n-1} F_A $ and $T \circ \beta = \tau_{\leq n-1} F_B$. This recovers the reformulation presented in \Cref{thm:framed-dom-wall-intro}.
\end{remark}

\begin{remark}
We present our statements as equivalences between different types of data, following the conventions of \cite{L09}. Formally, these statements assert an equivalence between the equivalence classes of objects in the corresponding $\infty$-groupoids. Given an appropriate Yoneda Lemma for symmetric monoidal $(\infty,n)$-categories, this is equivalent to asking that the corresponding $\infty$-groupoids are the same; see \cite[Remark 2.4.9]{L09}. 

In the case of a framed domain wall, \Cref{conj:main-conj} says that the following diagram is a pullback of $\infty$-groupoids
\[\begin{tikzcd}
	{\Fun^{\otimes}(\Bord_{n}(\Def^\fr_1),\Cc)} && {\Fun^{\otimes}(\Bord_{n-1}^{\fr},\iota_{n-1}(\Cc^{\to}_{(1)}))} \\
	\\
	{\Fun^{\otimes}(\Bord_{n}^\fr,\Cc)^{\times 2}} && {\Fun^{\otimes}(\Bord_{n-1}^\fr,\iota_{n-1}\Cc)^{\times 2}}
	\arrow[from=1-1, to=1-3]
	\arrow["{(\iota_A^*, \iota_B^*)}"', from=1-1, to=3-1]
	\arrow["{(S_*, T_*)}", from=1-3, to=3-3]
	\arrow["{\tau_{\leq{n-1}}^{\times 2}}", from=3-1, to=3-3]
\end{tikzcd}\]
where the top arrow is defined, by the ordinary cobordism hypothesis, by sending $Z$ to the oplax natural transformation $\beta$ such that $\beta(\pt_+) = Z(\Idom_+)$.
\end{remark}

\begin{remark}
We also propose a version of Conjecture \ref{conj:main-conj}, replacing the oplax arrow categories with lax arrow categories; nothing else is required to change.
\end{remark}

\begin{remark}
We state Conjecture \ref{conj:main-conj} in the fully extended setting for comparison with the Cobordism Hypothesis for defects. However, a version of Conjecture \ref{conj:main-conj} should also hold in the non-fully extended setting. The caveat here is that one must only allow objects to be 0-flag-stratified, 1-morphisms to be 1-flag-stratified, and so on. This condition is immediate in the fully extended setting.
\end{remark}

Before comparing \Cref{conj:main-conj} with the cobordism hypothesis for defects, we observe the following inductive formulation of \Cref{conj:main-conj}.

\begin{conjecture}\label{conj:main-conj-ind}
Let $\Def$ be an $n$-dimensional defect datum of length $k \geq 1$ and let $\Def'$ be its truncation to length $k-1$. Let $\Cc$ be a symmetric monoidal $(\infty,n)$-category and let $$Z' : \Bord_{n}(\Def') \to \Cc$$ be a symmetric monoidal functor. There is an equivalence between the following:
\begin{enumerate}
    \item Symmetric monoidal functors $Z:\Bord_{n}(\Def) \to \Cc$ extending $Z'$.
    \item Families of symmetric monoidal functors
    \begin{align*}
         \{\beta_{\lambda}:\Bord_{n-k}^{(X_k,\zeta_k)} \to \Cc^{\to}_{
         (k)}\}_{\lambda\in L_k}  
    \end{align*}
    such that $S_k(\beta_{\lambda}(\pt_{\widetilde{x}})) = Z'(\Icube{k-1}(\widetilde{q}_k(\widetilde{x}),s_k(\lambda)))$ and $T_k(\beta_{\lambda}(\pt_{\widetilde{x}})) = Z'(\Icube{k-1}(\widetilde{q}_k(\widetilde{x}),t_k(\lambda)))$ for all $\lambda \in L_k$ and $\widetilde{x}\in \widetilde{X}_k$.
\end{enumerate}
Moreover, this equivalence satisfies $Z(\kcube(\widetilde{x},\lambda)) = \beta_{\lambda}(\pt_{\widetilde{x}})$.
\end{conjecture}

\begin{remark}
The case $k=0$ may be regarded as the base case of this inductive formulation. There is then no lower-level defect datum $\Def'$. We interpret $\Bord_n(\Def')$ as the trivial bordism category, so that the functor $Z'$ carries no data, and set $\Cc^{\to}_{(0)} := \Cc$. In this case, the conjecture follows from the ordinary cobordism hypothesis, since an $(X_0\times L_0,\zeta_0\times L_0)$-structure on the point is equivalent to an $(X_0,\zeta_0)$-structure on the point and an element $\lambda \in L_0$.
\end{remark}

\begin{lemma}
Let $k\geq 1$ and suppose that \Cref{conj:main-conj} holds for all $n$-dimensional defect data of length at most $k-1$. Then, assuming the ordinary cobordism hypothesis, \Cref{conj:main-conj} holds for all $n$-dimensional defect data of length $k$ if and only if \Cref{conj:main-conj-ind} holds for all $n$-dimensional defect data of length $k$.
\end{lemma}

\begin{proof}
Let $\Def$ be an $n$-dimensional defect datum of length $k$ and let $\Def'$ be its truncation to length $k-1$. By assumption, specifying a symmetric monoidal functor $Z'\colon\Bord_n(\Def')\longrightarrow\Cc$ is equivalent to specifying the compatible families
\begin{equation*}
    \{
    \beta^{(i)}_\lambda\colon
    \Bord_{n-i}^{(X_i,\zeta_i)}
    \longrightarrow
    \Cc^\to_{(i)}
    \}_{\lambda\in L_i},
\end{equation*}
for $0\leq i\leq k-1$, appearing in \Cref{conj:main-conj}. Under this equivalence, $Z'\bigl(\Icube{i}(\widetilde x,\lambda)\bigr) = \beta^{(i)}_\lambda(\pt_{\widetilde x})$ for every $\lambda\in L_i$ and $\widetilde x\in\widetilde X_i$.

Consequently, by Conjecture \ref{conj:main-conj} applied to $\Def$, extending this system from length $k-1$ to length $k$ amounts to choosing a family
\begin{equation*}
    \left\{
    \beta_\lambda\colon
    \Bord_{n-k}^{(X_k,\zeta_k)}
    \longrightarrow
    \Cc^\to_{(k)}
    \right\}_{\lambda\in L_k}
\end{equation*}
whose source and target agree with the codimension-$(k-1)$ functors determined by $Z'$. By the ordinary cobordism hypothesis, these compatibilities are determined, up to contractible choice, by evaluation on points with $(X_k,\zeta_k)$-structure. Using the preceding description of $Z'$, they become
\begin{align*}
    S_k\bigl(\beta_\lambda(\pt_{\widetilde x})\bigr)
    =
    Z'\bigl(\Icube{k-1}(\widetilde q_k(\widetilde x),s_k(\lambda))\bigr) \quad{\text{and}}\quad 
    T_k\bigl(\beta_\lambda(\pt_{\widetilde x})\bigr)
    =
    Z'\bigl(\Icube{k-1}(\widetilde q_k(\widetilde x),t_k(\lambda))\bigr)
\end{align*}
for every $\lambda\in L_k$ and $\widetilde x\in\widetilde X_k$.

Thus, after identifying the lower-codimension system with $Z'$, the additional data appearing in \Cref{conj:main-conj} is exactly the data appearing in \Cref{conj:main-conj-ind}. Moreover, in both formulations, the correspondence with an extension $Z$ is characterized by $Z\bigl(\kcube(\widetilde x,\lambda)\bigr) = \beta_\lambda(\pt_{\widetilde x})$. It follows that the two conjectures are equivalent in length $k$.
\end{proof}

\subsection{Comparison with the cobordism hypothesis for defects}\label{sec:comparison}

We now compare Conjecture~\ref{conj:main-conj} with the cobordism hypothesis for defects. We first specialize the cobordism hypothesis with singularities of \cite[Theorem 4.3.11]{L09} to the prototypical defect structures introduced in Section~\ref{sec:def-structure}. The resulting formulation classifies defects inductively in the length of their underlying defect structure, in terms of $O(n-k)$-equivariant families of ambiently $n$-dualizable $k$-morphisms. We then show that, assuming the ordinary cobordism hypothesis, this classification is equivalent to Conjecture~\ref{conj:main-conj}.

\subsubsection{The cobordism hypothesis for defects}

First, consider a framed defect of codimension $k$. This corresponds to the defect datum given by the framed defect tangential structure of Example \ref{ex:framings-n} and the codimension-$k$ defect labeling of Example \ref{ex:codim-k-defect-labeling}. We denote this defect datum by $\Def^{\fr}_k$. The $k$-cube $\kcube$ admits a canonical $\Def_{k}^{\fr}$ structure, and we denote the $k$-cube with this structure by $\kcube_+$.

In light of the $O(n-k)$-action provided by \Cref{cor:O(n-k)-action}, the cobordism hypothesis with singularities applied to the framed defect of codimension $k$ can be restated as follows.

\begin{hypothesis}[{Lurie \cite[Theorem 4.3.11]{L09}}]\label{hyp:framed-def-2}
Let $\Cc$ be a symmetric monoidal $(\infty,n)$-category. Then evaluation on the  $k$-morphism given by the framed $k$-cube $\kcube_+$ yields an $O(n-k)$-equivariant equivalence
\begin{align*}
    \Fun^{\otimes}(\Bord_{n}(\Def^{\fr}_k),\Cc )\xrightarrow{\simeq} \DefCat_{k}^n(\Cc)^\sim
\end{align*}
of $\infty$-groupoids, where $\DefCat_{k}^n(\Cc)$ is the $(\infty,n-k)$-category of ambiently $n$-dualizable $k$-morphisms in $\Cc$ given in Definition~\ref{def:n-def-dual-cat}.
\end{hypothesis}

\begin{remark}
Setting $k=0$ in Hypothesis \ref{hyp:framed-def-2} yields the plain framed cobordism hypothesis. In this case, the $O(n)$-equivariance is immediate by definition: the $O(n)$-action on $\Def_{0}^n(\Cc)^\sim := (\Cc^{nd})^{\sim}$ is the one induced by the equivalence $\Fun^{\otimes}(\Bord_{n}^\fr,\Cc) \xrightarrow{\simeq} (\Cc^{nd})^\sim $.
\end{remark}

\begin{remark}
Since the equivalence of \Cref{hyp:framed-def-2} is implemented by evaluation on the framed $k$-cube $\Icube{k}_+$, the ordinary cobordism hypothesis implies that the following diagram
\begin{equation*}
    \begin{tikzcd}[row sep=small,column sep=small]
	{\Fun^{\otimes}(\Bord_n(\Def_k^{\fr}),\Cc)} && {\Def_{k}^{n}(\Cc)^\sim} \\
	\\
	{\Fun^{\otimes}(\Bord_n(\Def_{k-1}^{\fr}),\Cc)} && {\Def_{k-1}^{n}(\Cc)^\sim}
	\arrow["\simeq", from=1-1, to=1-3]
	\arrow["{\underline{\hspace{3mm}}\ \circ\ \iota_{A}}"', from=1-1, to=3-1]
	\arrow["{S_k}", from=1-3, to=3-3]
	\arrow["\simeq", from=3-1, to=3-3]
\end{tikzcd}
\end{equation*}
commutes. Here $\iota_{A}:\Bord_n(\Def_{k-1}^{\fr}) \to \Bord_n(\Def_k^{\fr})$ is the canonical inclusion corresponding to the framed codimension-$(k-1)$ defect labeled by the source of the framed codimension-$k$ defect. Replacing $S$ with $T$ and $A$ with $B$ results in another commutative diagram.
\end{remark}

We now extend Hypothesis \ref{hyp:framed-def-2} to the setting of arbitrary defect tangential structures of length $k$. Let $\defect = (\SingDat,\Lab)$ be an $n$-dimensional defect datum of length $k$ as in Definition \ref{def:defect-datum} and let $\defect'$ be its truncation to an $n$-dimensional defect datum of length $k-1$. Recall from Section \ref{sec:standard-cubes}, that $\widetilde{X}_k$ is the frame bundle of $\zeta_k \to X_k$ and $\tilde{q}_k:\widetilde{X}_k\to \widetilde{X}_{k-1}$ is the canonical map induced by the pullback diagram \eqref{eq:tang-str-pb-n}.

The following hypotheses are obtained by specializing the cobordism hypothesis with singularities to the class of defect structures introduced in Section~\ref{sec:def-structure}.

\begin{hypothesis}[{Lurie \cite[Theorem 4.3.11]{L09}}]\label{hyp:def-cob-hyp-main}
Let $\Cc$ be a symmetric monoidal $(\infty,n)$-category. Let $$Z' : \Bord_{n}(\Def') \to \Cc$$ be a symmetric monoidal functor. Then there is an $O(n-k)$-equivariant equivalence between
\begin{enumerate}
    \item Symmetric monoidal functors $Z:\Bord_{n}({\defect}) \to \Cc$ extending $Z'$ (here $O(n-k)$ acts via its action on the defect tangential structure $\SingDat$). 
    \item Families of ambiently $n$-dualizable $k$-morphisms
    \begin{equation*}
        \{f_{\tilde{x},\lambda}:  Z'(\Icube{k-1}(\tilde{q}_k(\tilde{x}),s_{k}(\lambda))) \to Z'(\Icube{k-1}(\tilde{q}_k(\tilde{x}),t_{k}(\lambda)))\}_{\widetilde{x}\in \widetilde{X}_k,\lambda \in L_{k}}
    \end{equation*}
    (here $O(n-k)$ acts on the family by its action on the groupoid of ambiently $n$-dualizable $k$-morphisms coming from \Cref{cor:O(n-k)-action}).
\end{enumerate}
 Moreover, this equivalence is implemented by evaluation on the family of disks $\{\kcube(\tilde{x},\lambda)\}_{\widetilde{x}\in \widetilde{X}_k,\lambda \in L_{k}}$.
\end{hypothesis}

Hypothesis \ref{hyp:def-cob-hyp-main} identifies extensions of $Z'$ with families of ambiently $n$-dualizable $k$-morphisms. Unpacking the $O(n-k)$-equivariance of this equivalence yields the following fixed-point formulation, which is the form of the hypothesis presented in \cite{L09}.

\begin{hypothesis}[{Lurie \cite[Theorem 4.3.11]{L09}}]\label{hyp:def-cob-hyp-main-2}
Let $\Cc$ be a symmetric monoidal $(\infty,n)$-category. Let $$Z' : \Bord_{n}(\Def') \to \Cc$$ be a symmetric monoidal functor. A symmetric monoidal functor $Z:\Bord_{n}(\Def) \to \Cc$ extending $Z'$ is equivalent to an $O(n-k)$-equivariant family
    \begin{equation*}
        \{f_{\tilde{x},\lambda}: Z'(\Icube{k-1}(\tilde{q}_k(\tilde{x}),s_k(\lambda))) \to Z'(\Icube{k-1}(\tilde{q}_k(\tilde{x}),t_k(\lambda)))\}_{\widetilde{x}\in \widetilde{X}_k,\lambda \in L_{k}}
    \end{equation*}
of ambiently $n$-dualizable $k$-morphisms.
\end{hypothesis}

\begin{remark}
The hypotheses above admit a more intrinsic $\infty$-categorical formulation. Inductively, the space of symmetric monoidal functors out of a defect bordism category is obtained from the spaces
\[
\mathrm{Map}_{O(n-i)}(\widetilde{X}_i\times L_i,\DefCat_i^n(\Cc)^\sim)
\]
by an iterated pullback construction using the corresponding source and target maps. Compare with the discussion of equivalent formulations of the ordinary cobordism hypothesis in \cite[Remark 2.4.9]{L09}.
\end{remark}

\subsubsection{The comparison theorem}

We now prove our main comparison theorem, relating the cobordism hypothesis for defects to the classification of defects in terms of oplax natural transformations. 

\begin{theorem}\label{thm:main-comparison}
Assuming the ordinary cobordism hypothesis, the cobordism hypothesis for defects (Hypothesis \ref{hyp:def-cob-hyp-main-2})  holds if and only if Conjecture \ref{conj:main-conj} holds.
\end{theorem}

\begin{remark}
In the case of the defect datum $\Def_{1}^{\fr}$ describing a framed domain wall, \Cref{thm:main-comparison} recovers  \Cref{thm:framed-dom-wall-intro} from the introduction.    
\end{remark}

\begin{proof}[Proof of Theorem \ref{thm:main-comparison}]
Fix a dimension $n$. We will give a proof by induction on the length $k$ of an $n$-dimensional defect datum $\defect = (\SingDat,\Lab)$.

\textbf{Base case:} In the case $k=0$, Hypothesis \ref{hyp:def-cob-hyp-main-2} is the plain cobordism hypothesis applied to the tangential structure $(L_0 \times X_0,L_0 \times \zeta_0)$. Recall that an ambiently $n$-dualizable $0$-morphism in $\Cc$ is just an $n$-dualizable object in $\Cc$.

On the other hand, Conjecture \ref{conj:main-conj} is trivial for $k=0$; an $(X_0,\zeta_0)$-structure on a bordism $M$ along with an $L_0$ labeling is equivalent to an $(X_0\times L_0,\zeta_0\times L_0)$-structure on $M$. A symmetric monoidal functor $Z:\Bord_{n}^{(L_0\times X_0,L_0 \times \zeta_0)} \to \Cc$ is equivalent to a family of symmetric monoidal functors $\{\beta_{\lambda} : \Bord_{n}^{(X_0,\zeta_0)}\to \Cc\}_{\lambda\in L_0}$ via the assignment $\beta_{\lambda}(M) = Z(M)$, where the labeling on $M$ is the constant labeling at $\lambda \in L_0$.

Since Hypothesis \ref{hyp:def-cob-hyp-main-2} is an instance of the plain cobordism hypothesis and  Conjecture \ref{conj:main-conj} is trivial, these statements are both true, having assumed the plain cobordism hypothesis.

\textbf{Inductive step:} Let $k\leq n$ and suppose that Theorem \ref{thm:main-comparison} is true for all $i < k$. We must show that Conjecture \ref{conj:main-conj-ind} and Hypothesis \ref{hyp:def-cob-hyp-main-2} are equivalent for a defect datum $\defect$ of length $k$. Let $\defect'$ be the truncation of $\Def$ to an $n$-dimensional defect datum of length $k-1$ and $Z':\Bord_{n}(\Def') \to \Cc$ a symmetric monoidal functor. Conjecture \ref{conj:main-conj-ind} and Hypothesis \ref{hyp:def-cob-hyp-main-2} give two different classifications of extensions of $Z'$ to a symmetric monoidal functor out of $\Bord_{n}(\Def)$. We must show that the data used to classify the extensions are equivalent.

Conjecture \ref{conj:main-conj-ind} classifies an extension of $Z'$ in terms of
\begin{enumerate}[label=(\arabic*)]
    \item A family of symmetric monoidal functors
    \begin{align*}
         \{\beta_{\lambda}:\Bord_{n-k}^{(X_k,\zeta_k)} \to \Cc^{\to}_{(k)}\}_{\lambda\in L_k}  
    \end{align*}
    such that $S_k(\beta_{\lambda}(\pt_{\tilde{x}})) = Z'(\Icube{k-1}(\tilde{q}_k(\tilde{x}),s_k(\lambda)))$ and $T_k(\beta_{\lambda}(\pt_{\tilde{x}})) = Z'(\Icube{k-1}(\tilde{q}_k(\tilde{x}),t_k(\lambda)))$.
\end{enumerate}
The plain cobordism hypothesis gives an equivalence between $(1)$ and the following:
\begin{enumerate}[label=(\arabic*),start=2]
    \item An $O(n-k)$-equivariant family $$\{f_{\tilde{x},\lambda}:  Z'(\Icube{k-1}(\tilde{q}_k(\tilde{x}),s_k(\lambda))) \to Z'(\Icube{k-1}(\tilde{q}_k(\tilde{x}),t_k(\lambda))) \}_{\widetilde{x}\in \widetilde{X}_k,\lambda \in L_{k}}$$ of $(n-k)$-dualizable objects in $\Cc^{\to}_{(k)}$.
\end{enumerate}
The $(n-k)$-dualizable objects of $\Cc_{(k)}^\to$ are classified in \cite{JFS15}, see \Cref{thm:JFS-oplax-dual}; $f_{\tilde{x},\lambda}$ is $(n-k)$-dualizable in $\Cc^{\to}_{(k)}$ if and only if $f_{\tilde{x},\lambda}^\sharp$ is $(n-k)$-times right-adjunctible in $\Cc$. Since the source and target of $f_{\tilde{x},\lambda}^\sharp$ are ambiently $n$-dualizable in $\Cc$, by Theorem \ref{thm:main-dual-result}, the $k$-morphism $f_{\tilde{x},\lambda}^\sharp$ is $(n-k)$-times right-adjunctible in $\Cc$ if and only if $f_{\tilde{x},\lambda}^\sharp$ is ambiently $n$-dualizable in $\Cc$. So the data of $(2)$ is equivalent to
\begin{enumerate}[label=(\arabic*),start=3]
\item An $O(n-k)$-equivariant family $$\{f_{\tilde{x},\lambda}^\sharp:  Z'(\Icube{k-1}(\tilde{q}_k(\tilde{x}),s_k(\lambda))) \to Z'(\Icube{k-1}(\tilde{q}_k(\tilde{x}),t_k(\lambda))) \}_{\widetilde{x}\in \widetilde{X}_k,\lambda \in L_{k}}$$ of ambiently $n$-dualizable $k$-morphisms in $\Cc$.
\end{enumerate}
This is the data used to classify an extension of $Z'$ in Hypothesis \ref{hyp:def-cob-hyp-main-2}. Since the data in (1) and (3) are equivalent, the two classifications of extensions of $Z'$ agree. This completes the inductive step and the proof.
\end{proof}

\begin{corollary}\label{cor:general-intro-thm-1}
Assume the ordinary cobordism hypothesis and the cobordism hypothesis for defects. Then \Cref{conj:main-conj} holds: for every $n$-dimensional defect datum $\Def$ of length $k$ and every symmetric monoidal $(\infty,n)$-category $\Cc$, symmetric monoidal functors $Z\colon\Bord_n(\Def)\to\Cc$ are equivalent to compatible families of symmetric monoidal functors $\{\beta_\lambda\}$ as in \Cref{conj:main-conj}.
\end{corollary}

\begin{proof}
Immediate from \Cref{thm:main-comparison}.
\end{proof}

\begin{remark}
Let $\Def_1$ be the defect datum corresponding to the domain wall labeling of \Cref{ex:dom-wall-labels} and no tangential structures (i.e., unoriented manifolds and strata). Then \Cref{cor:general-intro-thm-1}, applied to $\Def_1$, recovers \Cref{intro-thm-1} from the introduction. For a general $\Def$, it extends that statement to prototypical defects of arbitrary codimension, equipped with general tangential structures and labeling systems.
\end{remark}

\subsection{Extension to general singularities}\label{sec:conical}

We conclude by extending the preceding reformulation to the general singularities considered in \cite[Section~4.3]{L09}. The resulting statement has the following form: extending a field theory across a new class of singularities is equivalent to specifying a symmetric monoidal oplax natural transformation from the trivial theory to the theory obtained by evaluating the original field theory on the links of the singularities. We first give a brief account of Lurie's singularity data, then recall the cobordism hypothesis with singularities, and finally state this reformulation and prove its equivalence to the cobordism hypothesis with singularities.

\subsubsection{Singularity data and links}

Lurie's definition of a singularity datum \cite[Definition Sketch~4.3.2]{L09} is inductive. We do not recall the full definition here; instead, we describe the inductive step that contains the features needed for our application. Suppose that $\mathcal S'$ is an $n$-dimensional singularity datum of length $k-1$. An extension of $\mathcal S'$ to a singularity datum $\mathcal S$ of length $k$ is specified by a topological space $X_k$, a vector bundle $\zeta_k\to X_k$ of rank $n-k$, and a fiber bundle $p_k:E_k\to X_k$ equipped with a compatible family of $\mathcal S'$-structures on its fibers. For each $x\in X_k$, the fiber $E_x$ is a compact $(k-1)$-dimensional $\mathcal S'$-manifold whose stabilized tangential structure involves the complementary directions $\zeta_{k,x}\oplus\RR$. The fiber $E_x$ specifies the link of the singularity corresponding to $x$.

An $\mathcal S$-manifold has a codimension-$k$ stratum equipped with an $(X_k,\zeta_k)$-structure. Locally, a neighborhood of this stratum is obtained fiberwise by taking the cone on the corresponding links. Here, the component of $\RR$ in the stabilized tangent bundle corresponds to the radial direction of the cone. Thus, the link records how the codimension-$k$ stratum meets all of the singularities already specified by $\mathcal S'$.

The defect structures introduced in Section~\ref{sec:def-structure} arise as a special case of this construction. The link of a framed prototypical codimension-$k$ stratum is the sphere $S^{k-1}$ equipped with the flag
\begin{equation*}
S^0\subset S^1\subset\cdots\subset S^{k-1},
\end{equation*}
with labels determined by the iterated source and target maps. Thus, our prototypical defects correspond to a particular globular choice of link, whereas Lurie's definition allows the link to be an arbitrary compact $\mathcal S'$-manifold.

\begin{example}
When $k=1$, the links are compact zero-dimensional manifolds. The link of a domain wall is $S^0$: its two points correspond to the two sides of the wall, and their structures record the source and target bulk data. By contrast, a one-point link produces manifolds with boundary.
\end{example}

\subsubsection{The cobordism hypothesis with singularities}

Let $\widetilde X_k\to X_k$ denote the frame bundle of $\zeta_k$. A frame $\widetilde x$ of $\zeta_{k,x}$ identifies $\zeta_{k,x}$ with $\mathbb R^{n-k}$ and determines an object
\begin{equation*}
E_{\widetilde x}\in\Omega^{k-1}\Bord_n(\mathcal{S}')
\end{equation*}
represented by the link $E_x$ with its induced $\mathcal S'$-structure. The assignment $\widetilde x\longmapsto E_{\widetilde x}$ is $O(n-k)$-equivariant. Moreover, the cone on $E_{\tilde{x}}$ defines a $1$-morphism $C(E_{\tilde{x}}):\varnothing \to E_{\widetilde x}$ in $\Omega^{k-1}\Bord_n(\mathcal{S})$, or equivalently a $k$-morphism $C(E_{\tilde{x}}):\varnothing \to E_{\widetilde x}$ in $\Bord_n(\mathcal{S})$.

The following is the inductive form of the cobordism hypothesis with singularities. We state the dualizability condition explicitly; in Lurie's formulation, the target category has duals, so this condition is automatic.

\begin{hypothesis}[{Lurie \cite[Theorem~4.3.11]{L09}}]\label{hyp:general-singular-cobordism}
Let $\Cc$ be a symmetric monoidal $(\infty,n)$-category and let $Z'\colon\Bord_n(\mathcal{S}')\longrightarrow\Cc$
be a symmetric monoidal functor. There is an equivalence between the following data:
\begin{enumerate}
\item Symmetric monoidal functors $Z\colon\Bord_n(\mathcal S)\longrightarrow\Cc$
extending $Z'$.

\item $O(n-k)$-equivariant families of ambiently $n$-dualizable $k$-morphisms in $\Cc$ of the form
\begin{equation*}
    \left\{
    \eta_{\widetilde x}\colon
    \unit\longrightarrow Z'(E_{\widetilde x})
    \right\}_{\widetilde x\in\widetilde X_k}
\end{equation*}
\end{enumerate}
Moreover, this equivalence is implemented by evaluation on the family of cones parametrized by $\tilde{x}\in \widetilde{X}_k$: $Z \mapsto 
\{Z\bigl(C(E_{\widetilde x})\bigr)
\}_{\tilde{x}\in \widetilde{X}_k}$.
\end{hypothesis}

\subsubsection{Reformulation via oplax natural transformations}

The family of links determines a symmetric monoidal functor
\begin{equation*}
\operatorname{Link}_{\mathcal{S}}\colon
\Bord_{n-k}^{(X_k,\zeta_k)}
\longrightarrow
\Omega^{k-1}\Bord_n(\mathcal{S}').
\end{equation*}
Geometrically, if $M$ is a manifold with $(X_k,\zeta_k)$-structure given by the map $q\colon M\to X_k$, then
\begin{equation*}
\operatorname{Link}_{\mathcal S}(M) = M\times_{X_k}E_k,
\end{equation*}
equipped with its induced $\mathcal S'$-structure. In particular, $\operatorname{Link}_{\mathcal{S}}(\mathrm{pt}_{\widetilde{x}}) = E_{\widetilde{x}}$.

Given a symmetric monoidal functor $Z'\colon\Bord_n(\mathcal{S}')\longrightarrow\Cc$, we define the associated link theory to be the composite
\begin{equation*}
Z'_{\operatorname{Link}}
:=
\bigl(\Omega^{k-1}Z'\bigr)\circ\operatorname{Link}_{\mathcal S}\colon
\Bord_{n-k}^{(X_k,\zeta_k)}
\longrightarrow
\Omega^{k-1}\Cc.
\end{equation*}
Let
\begin{equation*}
\underline{\unit}\colon
\Bord_{n-k}^{(X_k,\zeta_k)}
\longrightarrow
\Omega^{k-1}\Cc
\end{equation*}
denote the trivial symmetric monoidal functor.

\begin{conjecture}\label{conj:general-singular-reformulation}
Let $\mathcal S$ be an $n$-dimensional singularity datum of length $k$ extending a singularity datum $\mathcal S'$ of length $k-1$, and let $Z'\colon\Bord_n(\mathcal{S}')\longrightarrow\Cc$ be a symmetric monoidal functor. There is an equivalence between the following data:
\begin{enumerate}
\item Symmetric monoidal functors $Z:\Bord_n(\mathcal{S})\longrightarrow\Cc$ extending $Z'$.

\item Symmetric monoidal oplax natural transformations $\beta \colon \underline{\unit} \Longrightarrow Z'_{\operatorname{Link}}$

\end{enumerate}
Moreover, this equivalence satisfies
\begin{equation*}
\beta(\mathrm{pt}_{\widetilde x})
=
Z\bigl(C(E_{\widetilde x})\bigr)
\colon
\unit\longrightarrow Z'(E_{\widetilde x})
\end{equation*}
in $\Omega^{k-1}\Cc$, for every $\widetilde x\in\widetilde X_k$.
\end{conjecture}

\begin{remark}
One may replace oplax natural transformations with lax natural transformations in the formulation of Conjecture \ref{conj:general-singular-reformulation}.
\end{remark}

\begin{remark}
For a prototypical domain wall, the link is $S^0$, with its two points carrying the source and target bulk data. If these bulk theories are denoted by $F_A$ and $F_B$, then the link functor is
\begin{align*}
    \operatorname{Link} : \Bord_{n-1}^{(X_1,\zeta_1)} &\to  \Bord_{n}^{(X_0,\zeta_0)} ,\\
    \pt_{\widetilde{x}} &\mapsto \pt_{\widetilde{q}(\widetilde{x}),A}^\vee \sqcup \pt_{\widetilde{q}(\widetilde{x}),B} ,
\end{align*}
and the associated link theory  $Z'_{\operatorname{Link}}$ is therefore equivalent to $\tau_{\leq n-1}F_A^\vee\otimes \tau_{\leq n-1}F_B$. Under the usual folding equivalence, a transformation $\underline{\unit}\Rightarrow \tau_{\leq n-1}F_A^\vee\otimes \tau_{\leq n-1}F_B$ is equivalent to an oplax natural transformation from $\tau_{\leq n-1}F_A$ to $\tau_{\leq n-1}F_B$. Thus, the formulation above recovers the reformulation for prototypical defects after folding.
\end{remark}

\begin{theorem}\label{thm:general-singular-comparison}
Assuming the ordinary cobordism hypothesis, Conjecture~\ref{conj:general-singular-reformulation} is equivalent to the cobordism hypothesis with singularities, Hypothesis~\ref{hyp:general-singular-cobordism}.
\end{theorem}

\begin{proof}
A symmetric monoidal oplax natural transformation $\beta\colon \underline{\unit} \Longrightarrow Z'_{\operatorname{Link}}$ is by definition a symmetric monoidal functor
\begin{equation*}
\beta: \Bord_{n-k}^{(X_k,\zeta_k)}
\longrightarrow
\bigl(\Omega^{k-1}\Cc\bigr)^{\rightarrow}
\end{equation*}
whose source is $\underline{\unit}$ and whose target is $Z'_{\operatorname{Link}}$. By the constructions of \cite[Section~5]{JFS15}, the oplax arrow category $(\Omega^{k-1}\Cc)^\to$ identifies with the symmetric monoidal subcategory of $\Cc_{(k)}^\to$ on which all iterated source and target functors landing below level $k-1$ are trivial.

By the ordinary cobordism hypothesis, such functors are classified by $O(n-k)$-equivariant families of $(n-k)$-dualizable objects in the oplax arrow category $\Cc_{(k)}^\to$ with the prescribed source and target: each $\widetilde x\in\widetilde X_k$  is assigned the $k$-morphism $\beta_{\widetilde x}\colon \unit\to Z'(E_{\widetilde x})$.

Both $\unit$ and $Z'(E_{\widetilde x})$ are ambiently $n$-dualizable $(k-1)$-morphisms in $\Cc$. By the classification of dualizable objects in the oplax arrow category, Theorem \ref{thm:JFS-oplax-dual}, $\beta_{\tilde{x}}$ is $(n-k)$-times right-adjunctible in $\Cc$. By \Cref{thm:main-dual-result}, we have that $\beta_{\tilde{x}}$ is ambiently $n$-dualizable in $\Cc$. Hence, the symmetric monoidal functor $\beta$ is classified by an $O(n-k)$-equivariant family of ambiently $n$-dualizable $k$-morphisms in $\Cc$.

This is precisely the data appearing in Hypothesis~\ref{hyp:general-singular-cobordism}. In both formulations, the correspondence is implemented by evaluation on the cone $C(E_{\widetilde x})$. Therefore, the two formulations are equivalent. 
\end{proof}

\section{Application: Reducibility for modules over $E_n$-algebras}\label{sec:En-application}

In this section we apply \Cref{thm:main-dual-result} to the higher Morita category and prove \Cref{thm:En-bimod} and \Cref{thm:En-mod}. In \Cref{sec:morita} we fix notation for the higher Morita category and recall the relative dualizability conditions of \cite{SSS26}, which detect one-sided iterated adjunctibility of its $1$-morphisms. In \Cref{sec:En-theorems} we combine these results with \Cref{thm:main-dual-result}.

\subsection{Dualizability in the higher Morita category}\label{sec:morita}

Throughout this section, $\Ss$ denotes a presentably symmetric monoidal
$(\infty,1)$-category. We recall as much of the higher Morita category $\Mor_{n}(\Ss)$ of
\cite{Sch14,Kar25,SSS26} as we need; our results use only the dualizability statements of \cite{GS18,SSS26}, and not the details of the construction, for which we refer the reader to those sources. This model is expected to be equivalent to Haugseng's combinatorial model \cite{Hau17}, a proof of which is the subject of work in progress by Scheimbauer--Steffens--\v{S}vraka
\cite{ScStSv}.

The higher Morita category $\Mor_{n}(\Ss)$ is a symmetric monoidal
$(\infty,n+1)$-category built from factorization algebras on $\RR^n$. Its
objects are the locally constant factorization algebras on $\RR^n$, equivalently, the $E_n$-algebras in $\Ss$ \cite[Section~5.4.2]{LHA}.
For $1\leq k\leq n$, the
$k$-morphisms of $\Mor_{n}(\Ss)$ are constructible pointless factorization
algebras on $\RR^n$ stratified by the flag of subspaces
\begin{equation*}
    \{x_1=\cdots=x_k=0\} \subset \{x_1=\cdots=x_{k-1}=0\} \subset \cdots
    \subset \{x_1=0\} \subset \RR^n .
\end{equation*}
This is the same as the local model of a prototypical codimension-$k$ defect, as in \Cref{sec:def-structure}. For $k<n$, the stratification has no marked points, and a $k$-morphism is the same as a factorization algebra on $\RR^n$ constructible with respect to this stratification. For $k=n$, the deepest stratum is the marked point $\{0\}$, and an $n$-morphism is a constructible \emph{pointless} factorization algebra in the sense of \cite{Kar25}; in particular, the object assigned to a disk centered at the origin is not equipped with a
pointing. The $(n+1)$-morphisms are morphisms of pointless factorization algebras.

We will need $1$-morphisms explicitly. A $1$-morphism $M\colon A\to B$ in $\Mor_{n}(\Ss)$ is a constructible pointless factorization algebra on $\RR^n$ stratified by a single hyperplane $\{x_1=0\}$. In the case $n=2$, we depict it as
\begin{equation*}
    \begin{tikzpicture}
\draw[defectpurple,thick,dashed] (1,0) -- (0,0) -- (0,2)--(1,2);
\filldraw[defectpurple,opacity=0.3] (1,0) -- (0,0) -- (0,2)--(1,2);
\draw[defectyellow,thick,dashed] (1,0) -- (2,0) -- (2,2)--(1,2);
\filldraw[defectyellow,opacity=0.3] (1,0) -- (2,0) -- (2,2)--(1,2);
\draw[very thick, green!85!black] (1,0) -- (1,2) node[pos=0.8,right,green!85!black] {$M$};
\node[defectpurple] at (0.4,0.6) {$A$};
\node[defectyellow] at (1.6,0.6) {$B$};
\end{tikzpicture}
\end{equation*}
Restricting to an open disk contained in $\{x_1<0\}$ gives an $E_n$-algebra $A\in\Ss$, the source of $M$, and restricting to an open disk in $\{x_1>0\}$ gives an $E_n$-algebra $B$, the target. The data assigned to a disk intersecting the hyperplane (in a single connected component) is the object of $\Ss$ underlying $M$, which we again denote by $M$. The structure maps for configurations of disks intersecting the hyperplane (in a single connected component) make $M$ an $E_{n-1}$-algebra, and those for configurations involving disks in the two half-spaces equip it with commuting actions of $A$ and $B$. Importantly, in the case $n=1$, the stratum $\{0\}$ is marked, and $M$ corresponds to a bimodule of $E_1$-algebras in the usual sense, rather than a pointed bimodule.

We refer to such a $1$-morphism as a \emph{bimodule of $E_n$-algebras}, or an \emph{$(A,B)$-bimodule} for short. When referring to a module over the underlying $E_1$-algebra of $A$, we use the term \emph{$E_1$-module}. 

It was shown in \cite{GS18} that $\Mor_{n}(\Ss)$ has $n$-duals; in particular every $E_n$-algebra $A$ is $n$-dualizable as an object of $\Mor_{n}(\Ss)$, and so determines an $n$-dimensional framed topological field theory $F_A$ valued in $\Mor_{n}(\Ss)$. A construction of these field
theories via factorization homology is given in \cite{Sch14}.

Since $\Mor_{n}(\Ss)$ is an $(\infty,n+1)$-category, the first genuinely restrictive dualizability condition occurs at level $n+1$. With pointings, \cite{GS18} showed that these conditions enforce every $(n+1)$-dualizable $E_n$-algebra $A$ to be trivial. In the pointless setting, a formulation of $(n+1)$-dualizability was conjectured by Lurie \cite{L09}, and proven in \cite{SSS26} for the pointless factorization algebra model (see \cite{BV26} for related work): an $E_n$-algebra $A$ is $(n+1)$-dualizable if and only if $A$ is dualizable as an $E_1$-module over the factorization homologies
\begin{equation*}
    \int_{S^{k-1}\times \RR \times D^{n-k}} A  
\end{equation*} 
for $0\leq k \leq n$.

Throughout this section we apply the results of \Cref{sec:dualizability} with the ambient symmetric monoidal $(\infty,N)$-category $\Cc$ taken to be $\Mor_{n}(\Ss)$ and $N=n+1$; the terms \emph{$(n+1)$-dualizable} and \emph{ambiently $(n+1)$-dualizable} always refer to this ambient category.

\begin{remark}[Field-theoretic interpretation]\label{rmk:morita-domain-walls}
Let $A$ and $B$ be $(n+1)$-dualizable, so that $F_A$ and $F_B$ extend to $(n+1)$-dimensional framed topological field theories. By \Cref{thm:main-dual-result}, a $1$-morphism $M\colon A\to B$ determines, via the cobordism hypothesis for framed domain walls, a framed domain wall between $F_A$ and $F_B$ exactly when $M$ is $n$-times right-adjunctible, or equivalently when $M$ satisfies any other one-sided iterated adjunctibility condition. We state the results below purely in terms of the equivalent dualizability conditions on $M$, but they may be read throughout as conditions for $M$ to define such a domain wall.
\end{remark}

\subsubsection{Partial dualizability in the higher Morita category}

In \cite[Theorem 4.1.3]{SSS26}, it is shown that $M$ is $n$-times right-adjunctible in $\mathrm{Mor}_n(\Ss)$ if and only if $M$ is dualizable as an $E_1$-module over the factorization homologies of certain stratified spaces with coefficients in $M$. We now recall these stratified spaces from \cite[Section~4]{SSS26}. For $1\leq k \leq n$, let $\widetilde{\RR}^k$ denote $\RR^k$ with the structure of a $k$-framed manifold with interface given by the codimension-one submanifold $\{x_1=0\}$. We think of the region $\{x_1<0\}$ as being labeled by $A$, the region $\{x_1>0\}$ as being labeled by $B$, and the interface as being labeled by $M$. The factorization homology of $\widetilde{\RR}^n$ is then the object of $\Ss$ underlying the 1-morphism $M$, which is also denoted $M$.

The submanifolds $D^k$ and $S^{k-1}$ inherit an interface from $\RR^k$. Denote the corresponding $k$-framed manifolds with interface by $\widetilde{D}^{k}$ and $\widetilde{S}^{k-1}$. Fix $\delta \in (0,1/2)$, and define the following $k$-framed submanifolds of $\widetilde{D}^k$ and $\widetilde{S}^{k-1}$:
\begin{alignat*}{3}
    \widetilde{D}^k_+ &:= \widetilde{D}^k \cap \{x_1 > -\delta\} &\qquad  \widetilde{D}^k_- &:= \widetilde{D}^k \cap \{x_1 < \delta \} \\
    \widetilde{S}^{k-1}_+ &:= \widetilde{S}^{k-1} \cap \{x_1 > -\delta \} &\qquad  \widetilde{S}^{k-1}_- &:= \widetilde{S}^{k-1} \cap \{x_1 < \delta \}.
\end{alignat*} 
For $k=1,2,3$, the stratified manifold $\widetilde{D}^k_-$ is depicted as follows:
\begin{equation*}
\begin{tikzpicture}
\draw[thick,defectyellow] (0,0) -- (0.5,0);
\draw[thick,defectpurple] (-1,0) -- (0,0);
\node[above,defectyellow] at (0.5,0) { $B$};
\node[above,defectpurple] at (-0.5,0) { $A$};
\filldraw[defectpurple]  (-1,0) circle (1pt);
\filldraw[green!85!black]  (0,0) circle (1pt);

\begin{scope}[xshift=4.5cm]
\filldraw[defectpurple,opacity=0.3] (0,1) arc (90:270:1);
\draw[thick,defectpurple] (0,1) arc (90:270:1) node[midway,left,defectpurple] {$A$} ;

\filldraw[defectyellow,opacity=0.3] (0,-1) -- (0.5,-1) -- (0.5,1) -- (0,1);
\draw[thick,defectyellow] (0,1) -- (0.5,1) node[above,defectyellow] {$B$};;
\draw[thick,defectyellow] (0,-1) -- (0.5,-1);
\draw[thick,defectyellow,dashed] (0.5,-1) -- (0.5,1);

\draw[very thick, green!85!black] (0,1) -- (0,-1);
\filldraw[green!85!black]  (0,1) circle (0.8pt);
\filldraw[green!85!black]  (0,-1) circle (0.8pt);
\end{scope}

\begin{scope}[xshift=9cm]

    %%%%% the back
    \filldraw[defectyellow,opacity=0.2] (0.5,1) -- (-0.3,1) arc (90:270:0.3 and 1) -- (0.5,-1) arc (270:90:0.3 and 1) ;
    \filldraw[defectpurple,opacity=0.2]  (-0.3,1) arc (90:270:1.3 and 1) arc (-90:90:0.3 and 1);
    \draw[thick,defectyellow,dashed]  (0.5,1) arc (90:270:0.3 and 1);
    \draw[very thick, green!85!black] (-0.3,1) arc (90:270:0.3 and 1);

    %%% the fill
    \filldraw[green,opacity=0.2] (0,0) arc (0:360:0.3 and 1); 
    \filldraw[defectyellow,opacity=0.1] (0.8,0) arc (0:360:0.3 and 1);

    %%%% the front
    \filldraw[defectyellow,opacity=0.3] (0.5,1) -- (-0.3,1) arc (90:-90:0.3 and 1) -- (0.5,-1) arc (-90:90:0.3 and 1) ;
    \filldraw[defectpurple,opacity=0.3]  (-0.3,1) arc (90:270:1.3 and 1) arc (-90:-270:0.3 and 1);
    \draw[thick,defectyellow,dashed] (0.5,1) arc (90:-90:0.3 and 1);
    \draw[defectyellow] (0.5,1) node[above,defectyellow] {$B$} -- (-0.3,1) ;
    \draw[defectyellow] (0.5,-1) -- (-0.3,-1);

    \draw[thick,defectpurple]  (-0.3,1) arc (90:270:1.3 and 1) node[midway,left,defectpurple] {$A$};
    \draw[very thick, green!85!black] (-0.3,1) arc (90:-90:0.3 and 1);

\end{scope}

\end{tikzpicture}
\end{equation*}

Since $\widetilde{S}^{k-1}_{\pm}$ is the boundary of $\widetilde{D}^k_{\pm}$, for $M:A\to B$ a 1-morphism in $\Mor_{n}(\Ss)$, the factorization homology of $\widetilde{S}^{k-1}_{\pm} \times \RR^{n-k+1}$ acts canonically on the factorization homology
\begin{align*}
    \int_{\widetilde{D}^{k}_{\pm} \times \RR^{n-k}} M_{A\to B} \simeq \int_{\widetilde{\RR}^n} M_{A\to B} \simeq M \ .
\end{align*}

It is convenient to name these factorization homologies. For
$1 \leq k \leq n$, set
\begin{equation}\label{eq:relative-algebras}
    \Aalg{k}{M} := \int_{\widetilde{S}^{k-1}_{-} \times \RR^{n-k+1}} M_{A\to B}
    \qquad\text{and}\qquad
    \Balg{k}{M} := \int_{\widetilde{S}^{k-1}_{+} \times \RR^{n-k+1}} M_{A\to B} \, ,
\end{equation}
so that $M$ is canonically a left module over $\Aalg{k}{M}$ and a right
module over $\Balg{k}{M}$. Note that $\widetilde{S}^0_-$ is the positively
framed point labeled by $A$ and $\widetilde{S}^0_+$ is the positively
framed point labeled by $B$, so that $\Aalg{1}{M} \simeq A$ and
$\Balg{1}{M} \simeq B$. For $k \geq 2$, however, these algebras depend on
$M$ as well as on $A$ and $B$.
 
\begin{definition}\label{def:n-dualizable-over}
Let $M \colon A \to B$ be a $1$-morphism in $\Mor_{n}(\Ss)$. We say that
$M$ is \emph{$n$-dualizable over $A$} if it is left dualizable as a left
module over $\Aalg{k}{M}$ for every $1\leq k \leq n$, and
\emph{$n$-dualizable over $B$} if it is right dualizable as a right module
over $\Balg{k}{M}$ for every $1\leq k \leq n$.
\end{definition}
 
With this terminology, \cite[Theorem 4.1.3]{SSS26} reads as follows: the $1$-morphism $M$ is $n$-times right-adjunctible in $\Mor_{n}(\Ss)$ if and only if it is $n$-dualizable over $B$, and $M$ is $R^{n-1}L$-adjunctible (that is, admits $(n-1)$ levels of right adjoints and one level of left adjoints; see \Cref{def:n-times-right-left}) if and only if it is $n$-dualizable over $A$. In \cite{SS23} it is shown that every one-sided iterated adjunctibility condition on $M$ is equivalent to one of these two
conditions.

\begin{example}\label{ex:n-equals-one-dualizable-over}
When $n = 1$, \Cref{def:n-dualizable-over} involves only $\Aalg{1}{M}\simeq A$ and $\Balg{1}{M}\simeq B$, and so reduces to the usual notions of dualizability of an $(A,B)$-bimodule as a left $A$-module and as a right $B$-module; see \cite[Section~4.6]{LHA}.
\end{example}

\subsection{Reducibility for modules over $E_n$-algebras}\label{sec:En-theorems}

We now prove \Cref{thm:En-bimod} and \Cref{thm:En-mod}, the latter we think of as a reducibility theorem for modules over $(n+1)$-dualizable algebras. Each of these results follows from applying \Cref{thm:main-dual-result} to a 1-morphism $M: A\to B$ in the higher Morita category $\Mor_{n}(\Ss)$, and simplifying the equivalent dualizability conditions on $M$ using the results of \cite{SSS26}.

\ThmEnBimod*

\begin{proof}
By \cite[Theorem~4.1.3]{SSS26} $M$ is $n$-dualizable over $A$ if and only if it is $R^{n-1}L$-adjunctible in $\Mor_{n}(\Ss)$, and $M$ is $n$-dualizable over $B$ if and only if it is $n$-times right-adjunctible in $\Mor_n(\Ss)$. Since $A$ and $B$ are assumed to be $(n+1)$-dualizable, by \Cref{thm:main-dual-result}, each of these conditions is equivalent to $M$ being ambiently $(n+1)$-dualizable.
\end{proof}

\begin{example}
If $n=1$, \Cref{thm:En-bimod} recovers a well-known result from \cite[Section 4.6]{LHA}. Recall that an $E_1$-algebra $A$ is 2-dualizable when it is smooth (dualizable as an $E_1$-module over $A^e = A \otimes A^{op}$) and proper (dualizable as an object in $\Ss$). For $A$ and $B$ smooth and proper $E_1$-algebras, \Cref{thm:En-bimod} recovers the statement that an $(A,B)$-bimodule $M$ is dualizable over $A$ if and only if $M$ is dualizable over $B$.
\end{example}

In the case that $B$ is trivial, that is, for a left module $M:A\to\unit$ over $A$, \Cref{thm:En-bimod} simplifies further.

\ThmEnMod*

\begin{proof}
By \Cref{thm:En-bimod}, $M$ is $n$-dualizable over $A$ if and only if $M$ is $n$-dualizable over $\unit$. By \cite[Proposition 4.5.1]{SSS26} $M$ is $n$-dualizable over $\unit$ if and only if the $E_{n-1}$-algebra underlying $M$ is $n$-dualizable in $\Mor_{n-1}(\Ss)$.
\end{proof}

\begin{example}
We now spell out the conditions in the case $n=2$. Let $A$ be a $3$-dualizable $E_2$-algebra and $M$ a left module over $A$. By \cite[Theorem~4.1.3]{SSS26}, $M$ is $2$-dualizable over $A$ if and only if $M$ is dualizable over $\Aalg{1}{M}$ and $\Aalg{2}{M}$. As noted above, $\Aalg{1}{M} \simeq A$.

To compute $\Aalg{2}{M}$, let $M^L\colon\unit\to A$ be a left adjoint of $M$. We define the $A$-linear endomorphism object of $M$ by
\begin{equation*}
    \mathrm{End}_A(M):=M\circ M^L
    \in\mathrm{End}_{\Mor_2(\Ss)}(\unit) \simeq \Mor_1(\Ss) .
\end{equation*}
In particular, $\mathrm{End}_A(M)$ is naturally an $E_1$-algebra in $\Ss$. Moreover, by the prescription of adjoints in \cite{GS18} and excision for stratified factorization homology
\cite{AFT17}, we have
\begin{equation*}
    \int_{\widetilde{S}^1_- \times \RR} M_{A\to \unit} := \int_{\begin{tikzpicture}[scale=0.3]
\filldraw[defectpurple,opacity=0.3] (0,1) arc (90:270:1) -- (0,-0.6) arc (270:90:0.6) -- (0,1);
\draw[defectpurple,dashed] (0,1) arc (90:270:1) ;
\draw[defectpurple,dashed] (0,0.6) arc (90:270:0.6);

\filldraw[defectyellow,opacity=0.3] (0,-1) -- (0.5,-1) -- (0.5,-0.6) -- (0,-0.6);
\filldraw[defectyellow,opacity=0.3] (0,1) -- (0.5,1) -- (0.5,0.6) -- (0,0.6);
\draw[dashed,defectyellow] (0,1) -- (0.5,1);
\draw[dashed,defectyellow] (0,0.6) -- (0.5,0.6);
\draw[dashed,defectyellow] (0.5,0.6) -- (0.5,1);

\draw[defectyellow,dashed] (0,-0.6) -- (0.5,-0.6);
\draw[defectyellow,dashed] (0,-1) -- (0.5,-1);
\draw[dashed,defectyellow] (0.5,-1) -- (0.5,-0.6);

\draw[very thick, green!85!black] (0,1) -- (0,0.6);
\draw[very thick, green!85!black] (0,-1) -- (0,-0.6);
\end{tikzpicture}} M_{A\to \unit} \simeq \int_{\begin{tikzpicture}[scale=0.4]
\draw[defectyellow,thick,dashed] (0.5,0) -- (0,0) -- (0,0.5) -- (0.5,0.5);
\filldraw[defectyellow,opacity=0.3] (0,0) rectangle (0.5,0.5);
\draw[defectpurple,thick,dashed] (0.5,0) -- (1,0) -- (1,0.5) -- (0.5,0.5);
\filldraw[defectpurple,opacity=0.3] (0.5,0) rectangle (1,0.5);
\draw[very thick, green!85!black] (0.5,0) -- (0.5,0.5);
\end{tikzpicture}} M_{A\to \unit} \otimes_{\int_{\begin{tikzpicture}[scale=0.3]
\draw[defectpurple,thick,dashed] (0.5,0) -- (0,0) -- (0,0.5) -- (0.5,0.5) -- (0.5,0);
\filldraw[defectpurple,opacity=0.3] (0,0) rectangle (0.5,0.5);
\end{tikzpicture}} M_{A\to \unit}} \int_{\begin{tikzpicture}[scale=0.4]
\draw[defectpurple,thick,dashed] (0.5,0) -- (0,0) -- (0,0.5) -- (0.5,0.5);
\filldraw[defectpurple,opacity=0.3] (0,0) rectangle (0.5,0.5);
\draw[defectyellow,thick,dashed] (0.5,0) -- (1,0) -- (1,0.5) -- (0.5,0.5);
\filldraw[defectyellow,opacity=0.3] (0.5,0) rectangle (1,0.5);
\draw[very thick, green!85!black] (0.5,0) -- (0.5,0.5);
\end{tikzpicture}} M_{A\to \unit} \simeq M \circ M^L  \, ,
\end{equation*}
where we have omitted framings from the figures.

\Cref{thm:En-mod} then says that $M$ is dualizable as an $E_1$-module over both $A$ and $\mathrm{End}_A(M)$ if and only if $M$ is 2-dualizable as an $E_1$-algebra. That is, $M$ is a smooth and proper algebra in $\Ss$.    
\end{example}

\begin{remark}
One can also apply \cite[Theorem 3.0.1]{SSS26} to reduce the condition that $M$ is $n$-dualizable in $\Mor_{n-1}(\Ss)$ to the condition that $M$ is dualizable as an $E_1$-module over the factorization homologies
\begin{equation*}
    \int_{S^{k-1}\times \RR \times D^{n-1-k}} M
\end{equation*}
for $0\leq k \leq n-1$.
\end{remark}

\subsubsection{Outlook: even higher Morita categories}\label{sec:even-higher}
It would also be interesting to extend these results to the ``even higher'' Morita categories, using the technology of \cite{JFS15}. Let $\Ss'$ be a suitable symmetric monoidal $(\infty,m)$-category with $m>1$, and consider the higher Morita $(\infty,n+m)$-category $\Mor_n(\Ss')$. After truncation, \cite[Theorem~4.1.3]{SSS26} characterizes the first $n$ stages of adjunctibility of a module. At present, however, there is no known analog that characterizes the additional $l$ stages required for a module $M$ to be $(n+l)$-times right-adjunctible, where $1\leq l\leq m-1$. Nevertheless, the following example suggests that one should still expect meaningful results for modules over $(n+l)$-dualizable $E_n$-algebras.

\begin{example}
Let $C$ and $D$ be finite semisimple tensor categories which are $3$-dualizable as objects of the $3$-category $\TensCat$ of finite tensor categories studied in \cite{DSPS18}, and let $M$ be a finite $(C,D)$-bimodule category. By \cite[Proposition~3.3.2]{DSPS18}, if $M$ is exact as a module category over $C$, then it is twice right-adjunctible. On the other hand, $M$ is ambiently $3$-dualizable if and only if its underlying finite category is semisimple. Since $C$ and $D$ are $3$-dualizable, Theorem~\ref{thm:main-dual-result} therefore implies that every finite $(C,D)$-bimodule category which is exact over $C$ is semisimple. This recovers a standard result in the theory of finite tensor categories; see \cite[Example~7.5.4]{EGNO15}.
\end{example}

\printbibliography

\par\bigskip
\noindent
\textsc{Technical University of Munich}\\
85748 Garching, Germany\\
\textit{Email address:}
\href{mailto:will.stewart@tum.de}{\texttt{will.stewart@tum.de}}

\end{document}